\documentclass{amsart}
\usepackage{amsfonts}
\usepackage{amsfonts,amsmath,amssymb}
\usepackage{mathrsfs,mathtools,stmaryrd,wasysym}
\usepackage{enumerate}
\usepackage{xspace,./mydef,./boldfonts} 
\usepackage{hyperref}
\usepackage{xcolor}
\usepackage{accents}

\newcommand{\utilde}[1]{\underaccent{\tilde}{#1}}
\newcommand{\TheTitle}{Finite element approximation of an obstacle problem for a class of integro--differential operators}
\newcommand{\ShortTitle}{Integro--differential operators}

\begin{document}

\title[\ShortTitle]{{\TheTitle}}

\author{Andrea Bonito}
\address{Department of Mathematics, Texas A\&M University, College Station, TX 77843, USA.}
\email{bonito@math.tamu.edu}
\author{Wenyu Lei}
\address{Department of Mathematics, Texas A\&M University, College Station, TX 77843, USA. Current address: SISSA - International School for Advanced Studies, Via Bonomea 265, 34136 Trieste, Italy}
\email{wenyu.lei@sissa.it}
\author{Abner J.~Salgado}
\address{Department of Mathematics, University of Tennessee, Knoxville, TN 37996, USA.}
\email{asalgad1@utk.edu}

\thanks{AB has been supported in part by NSF grant DMS-1254618. WL has been supported in part by NSF grant DMS-1254618. AJS has been supported in part by NSF grant DMS-1720213.}

\date{Draft version of \today.}

\begin{abstract}
We study the regularity of the solution to an obstacle problem for a class of integro--differential operators. The differential part is a second order elliptic operator, whereas the nonlocal part is given by the integral fractional Laplacian. The obtained smoothness is then used to design and analyze a finite element scheme.
\end{abstract}

\maketitle


\section{Introduction}
\label{sec:intro}

Let $\Omega \subset \mathbb R^d$, $d=1,2,3$, be an open bounded set with boundary $\partial\Omega$. We consider the following obstacle problem: given $f:\Omega \to \Real$, an obstacle $\chi :\overline\Omega \to \Real$ such that $\chi < 0$ on $\partial \Omega$, and a drift $\bbeta:\Omega \to \Real^d$, we want to find $u: \Rd \to \Real$ satisfying
\begin{equation}
\label{eq:minprob}
  \min\left\{ \bit Lu + \bbeta \cdot \GRAD w + \Laps u -f, u - \chi \right\} = 0, \text{ in } \Omega, \qquad u = 0, \text{ in } \complement\Omega.
\end{equation}
Here $\bit\in \polZ_2$; $\complement \Omega$ denotes the complement of $\Omega$; $L$ is a uniformly elliptic, divergence form, and symmetric second order differential operator
\begin{equation}
\label{eq:defofdiffop}
  L w = -\DIV(A\GRAD w) + cw,
\end{equation}
with sufficiently smooth coefficients (more precise conditions will be imposed later); and $\Laps$ with $s\in(0,1)$ denotes the integral fractional Laplacian, \ie
\begin{equation}
\label{eq:defofLaps}
  \Laps w(x) = c_{d,s} \text{p.v.} \int_{\Rd} \frac{w(x)-w(y)}{|x-y|^{d+2s}} \diff y, \qquad c_{d,s} = \frac{2^{2s} s \Gamma(s+\frac{d}{2})}{\pi^{d/2}\Gamma(1-s)},
\end{equation}
where p.v. stands for principal value.

The main motivation to study problem \eqref{eq:minprob} is its relevance in the context of perpetual American options under L\'evy processes (cf. \cite{BL02}). In one dimensional space ($d=1$), the solution $u$ in \eqref{eq:minprob} (but defined in $\Real$ instead of $\Omega$) is the rational price of a perpetual American option against the log-price of the stock assumed to  follow a L\'evy process whose infinitesimal generator is given by $\bit L+\bbeta \cdot \nabla+(-\Delta)^s$. In this context, the non-negative obstacle function $\chi$ is referred to as the payoff function; see \cite[Section 6]{BL02}. 
When $d>1$, problem \eqref{eq:minprob} (again in $\Real^d$ instead of $\Omega$) models multiple assets (cf. \cite{MR1459060}). 
For completeness, we point out that the jump process considered in this paper is a special case of a more general jump processes called tempered stable process.
For the latter, the integral fractional Laplacian in \eqref{eq:minprob} is replaced by a convolution in $\Real^d$ between $u$ and the kernel function
\[
	K(x) = C_0
	\begin{dcases}
	\frac{e^{-C_1|x|}}{|x|^{d+2s}},& |x|<0,\\
	\frac{e^{-C_2|x|}}{|x|^{d+2s}},& |x|>0,\\
	\end{dcases}
\]
where $C_0>0$ and $C_1,C_2\ge 0$. The process is symmetric if $C_1=C_2$ and reduces to the integral fractional Laplacian when $C_1=C_2=0$. We also note that to account for the fact that the original American option pricing problem is defined on the whole space $\Real^d$, one should analyze the so-called localization error between the solution of problem \eqref{eq:minprob} and the solution to the corresponding problem in $\Real^d$. These considerations are out of the scope of this work and we refer to \cite{MR2073930} for the analysis in the one dimensional case with $C_1,C_2>0$.

The goal of this paper is to obtain a finite element approximation to the solution of problem \eqref{eq:minprob} together with the corresponding \emph{a priori} error estimates in the energy space. Since these error estimates rely on the knowledge of the smoothness of the solution, we shall first study the regularity of the variational formulation of problem \eqref{eq:minprob}. Moreover, the nature of the operator at hand depends heavily on the particular values of $\bit$, $\bbeta$, and $s$ to be used, we address the following three different cases:
\begin{enumerate}[A.]
  \item \label{case:frac} \emph{Purely fractional diffusion}: $\bit = 0$, $\bbeta=\mathbf 0$, and $s \in (0,1)$. This corresponds to the obstacle problem for the integral fractional Laplacian.
  
  \item \label{case:fracdrift} \emph{Fractional diffusion with drift}: $\bit =0$, $\bbeta\neq\mathbf 0$, and $s \in [\tfrac12,1)$. In this case, the fractional power is restricted to keep the diffusive part dominant; see Proposition~\ref{prop:drift}.
  
  \item \label{case:diff} \emph{Integro--differential operator}: $\bit =1$ and $s \in (0,1)$.

\end{enumerate}

We remark that the regularity of the solution in Case~\ref{case:frac} has been already studied in \cite{BurkovskaGunzburger} and \cite{BNS}. 
To show the regularity of the result in the remaining cases, the main technique that we shall employ is based on penalizing the violation of the obstacle constraint, much in the spirit of the techniques presented in \cite[Section IV.2]{KinderlehrerStampacchia} and \cite[Section 1.3]{MR1009785}.
We derive regularity estimates for the unconstrained linear problem, which are instrumental to obtain a uniform regularity estimate for the solutions to the penalized problems. Passing to the limit when the penalization parameters tends to zero, we deduce the regularity of the solution to the obstacle problem. 
Since this is critical for the analysis of the proposed numerical method, we also show that the solution to the obstacle problem is continuous and that as a consequence, the so-called complementarity conditions are satisfied.

One of the main issues in the finite element approximation of the obstacle problem \eqref{eq:minprob} is the efficient approximation of the integral fractional Laplacian. We refer to \cite{MR3679865,MR3096457,BLP17}, see also the survey \cite{Bonito2018}, for different approaches. Unlike \cite{BurkovskaGunzburger,BNS}, here we use the method from \cite{BLP17,MR3671494}, \ie we build a numerical scheme based on the Dunford--Taylor integral representation of the bilinear form associated with the action of the integral fractional Laplacian operator; see Section~\ref{sub:DunfordTaylorreview} for a review of this approach. Adapting this technique to our case of interest induces a consistency error in the discretization of a variational inequality. We handle this via a Strang-type argument allowing us to derive rates of convergence in the energy error.

The outline of the paper is as follows. In Section~\ref{sec:notation} we set notation, introduce differential and integral operators, provide a weak formulation of \eqref{eq:minprob}, and show some of its immediate properties. In Section~\ref{sec:regularity} we study the regularity of the solution, the so--called Lagrange multiplier, and the validity of the complementarity conditions. Section~\ref{sec:FEM} provides the finite element algorithm and its error analysis as well. A detailed numerical implementation and numerical tests are provided in Section~\ref{sec:numex}.

\section{Notation and preliminaries}
\label{sec:notation}

In this work $\Omega \subset \Rd$ is a bounded domain with Lipschitz boundary $\partial\Omega$
(we may assume more on $\Omega$ if necessary).
Whenever we write $a \preceq b$ we mean that $a \leq C b$ for a nonessential constant $C$ that might change from line to line. As usual, $a \succeq b$ means $b \preceq a$; $a \asymp b$ means $a \preceq b \preceq a$.
Also, for any real number $a$, the notation $a^-$ henceforth stands for any real number strictly smaller than $a$.

For a normed space $X$, we denote by $X'$ and $\| \cdot \|_X$ its dual and norm, respectively. By $\scl\cdot,\cdot\scr_{X',X}$ we  denote the duality pairing. Unless explicitly stated, $X'$ is always equipped with the operator norm. In the case where $X$ is an inner product space, we denote by $(\cdot,\cdot)_X$ its inner product.

\subsection{Sobolev spaces on domains}
\label{sub:Sobolev}

The standard $\Ldeux$ and $H^m(\Omega)$ function spaces, $m \in \polN$, are normed in the usual way. 
We recall that $\Hunz$ is the closure in $\Hun$ of $C_0^\infty(\Omega)$ --- the space of compactly supported in $\Omega$ and infinitely differentiable functions. Owing to the Poincar\'e inequality, we have that
\[
  \| w \|_{\Hunz} := \| \GRAD w \|_{\Ldeux},
\]
is an equivalent norm on $\Hunz$.

Since $\Hun \subset \Ldeux$ and $\Hunz \subset \Ldeux$ are compatible pairs, we  define the fractional Sobolev spaces by interpolation using the real method 
$$
H^t(\Omega) := (\Hun,\Ldeux)_{1-t,2} \quad \textrm{and } \dot H^t(\Omega) := (\Hunz,\Ldeux)_{1-t,2}, \quad \text{for }t\in (0,1).
$$ 
By convention, $H^0(\Omega)=\dot H^0(\Omega)=L^2(\Omega)$ and $\dot H^1(\Omega)=H^1_0(\Omega)$. However, since the definition of the integral fractional Laplacian \eqref{eq:defofLaps} involves integration over the whole space, we need to introduce yet another family of function spaces. For $t \in [0,2]$ we define
\[
  H^t(\Rd) := \left\{ w : \Rd \to \Real \colon \| w \|_{H^t(\Rd)} < \infty \right\}, \  
    \| w \|_{H^t(\Rd)}^2 := \int_\Rd (1+|\xi|^t) |\calF(w)(\xi)|^2 \diff \xi,
\]
where $\calF$ denotes the Fourier transform.
Furthermore, for any bounded domain $D \subset \Real^d$ and $w:D \to \Real$ we denote by $\widetilde w$ its extension by zero to $\complement D$. Notice that this operator depends on $D$ which may change depending on the context.
However, we decided not to indicate the dependency on $D$ whenever no confusion is possible in order to alleviate the notation.
 With this we define, for $t\in [0,2]$,
\[
  \tHr:= \left\{ w \in \Ldeux \colon \widetilde w \in H^t(\Rd) \right\}, \qquad \| w \|_{\tHr} := \| \widetilde w \|_{H^t(\Rd)}.
\]
We finally set $H^{-t}(\Omega) = (\tHr)'$. 

\begin{remark}[equivalent norm]
\label{r:ths-norm}
A variant of the arguments in the Peetre--Tartar lemma \cite[Lemma A.38]{ERNGuermond} guarantees that the semi-norm 
\[
w \mapsto |w|_{\tHr} := \left( \int_\Rd |\xi|^t |\calF(w)(\xi)|^2 \diff \xi \right)^{1/2}
\] 
is an equivalent norm of $\tHr$. 
\end{remark}

\begin{remark}[norm equivalence for Lipschitz domains]\label{r:norm-equivalence}
For $t\in [0,1]$, it is known that $\dot H^t(\Omega)$ and $\tHr$ are both interpolation scales and coincide (cf. \cite[Lemma 4.11]{CHM15}). We note that these two spaces are also equivalent when $t\in(1,\tfrac32)$ and the norm equivalence constants depend on $\Omega$. This is because $\Omega$ is Lipschitz so that $\tHr=H^t(\Omega)\cap H^1_0(\Omega)=\dot H^t(\Omega)$ when $t\in[1,\tfrac32)$ (cf. \cite[Remark 3.1]{BLP17}). 
\end{remark}

\subsection{Differential and integral operators}
\label{sub:diffintops}

We can now give a proper interpretation to the building blocks of problem \eqref{eq:minprob}.

We begin with the second order operator. We let $A \in C^{0,1}(\overline\Omega,\polS^d)$, where $\polS^d$ is the space of symmetric $d \times d$ matrices, be uniformly bounded and positive definite, \ie there exist constants $a_0, a_1 > 0$ such that
\[
  a_0 |\bv|^2 \le \bv\Tr A(x) \bv \leq a_1 |\bv|^2, \quad \forall \bv \in \Rd, \ \forall x \in \overline\Omega.
\]
In addition, we assume that $c \in C^{0,1}(\overline\Omega)$ is nonnegative. With these assumptions we have that the operator $L:\Hdeux \cap \Hunz \to \Ldeux$ generates the bilinear form
\[
  \calL(v,w) = \int_\Omega \left( \GRAD w\Tr A(x)\GRAD v + c(x)vw \right) \diff x,
\]
which is bounded and coercive on $\Hunz$.

We now study drift on fractional Sobolev spaces. Let $\bbeta \in C^1(\overline\Omega,\Rd)$ be solenoidal, \ie $\DIV \bbeta = 0$. We define, for $v,w \in C_0^\infty(\Omega)$ the bilinear form
\begin{equation}
\label{eq:defofcalD}
  \calD(v,w) = \int_\Omega \bbeta(x) \cdot \GRAD v w \diff x
\end{equation}
and study the properties of $\calD$ next.
\begin{proposition}[drift]
\label{prop:drift}
Let $\bbeta \in C^1(\overline\Omega,\Real^d)$ be solenoidal, \ie $\DIV \bbeta = 0$. Then, for $v \in \tHr$ with $t\in[\tfrac12, 1]$ we have that
\[
  \| \bbeta \cdot \GRAD v \|_{H^{-t}(\Omega)} \preceq \| \bbeta \|_{L^\infty(\Omega,\Real^d)} \| v \|_\tHr.
\]
Moreover, the bilinear form $\calD$, defined in \eqref{eq:defofcalD}, extends continuously to $\tHr \times \tHr$. This, in particular, implies that
\begin{equation}\label{e:beta_div_free_cons}
  \calD(v,v) = 0, \quad \forall v \in \tHr.
\end{equation}
\end{proposition}
\begin{proof}
The proof follows the argumentations in \cite{PengXXXX}.
We begin by assuming that $v \in\ C^\infty_0(\Omega)$, then we immediately conclude that $\bbeta \cdot \GRAD v \in \Ldeux$ with
\[
  \| \bbeta \cdot \GRAD v \|_\Ldeux \leq \| \bbeta \|_{L^\infty(\Omega,\Real^d)} \| v \|_\Hunz.
\]
Owing to the fact that $\bbeta$ is solenoidal, we also have that
\begin{align*}
  \| \bbeta \cdot \GRAD v \|_\Hmun &= \sup_{0 \neq w \in \Hunz} \frac{ \scl \bbeta \cdot  \GRAD v, w \scr_{\Hmun,\Hunz}}{ \| \GRAD w \|_\Ldeux} = 
    \sup_{0 \neq w \in \Hunz} \frac{ ( v, \bbeta \cdot \GRAD w )_\Ldeux}{ \| \GRAD w \|_\Ldeux} \\ &\leq \|\bbeta \|_{L^\infty(\Omega,\Real^d)} \| v \|_\Ldeux.
\end{align*}
Interpolating the previous two inequalities we then obtain that for $t\in [\tfrac12,1]$
\[
  \| \bbeta \cdot \GRAD v \|_{H^{-t}(\Omega)} \preceq \| \bbeta \|_{L^\infty(\Omega,\Real^d)} \| v \|_{\widetilde{H}^{1-t}(\Omega)} \preceq \| \bbeta \|_{L^\infty(\Omega,\Real^d)} \| v \|_\tHr,
\]
as we intended to show. The proof is complete upon noting that $C^\infty_0(\Omega)$ is dense in $\tHr$.
\end{proof}

We now proceed to define the integral fractional Laplacian given in \eqref{eq:defofLaps}. First, we note that for $w$ in the Schwartz space, this operator is defined by
\[
  \calF\left( \Laps w \right)(\xi) = |\xi|^{2s}\calF(w)(\xi),
\]
Moreover, it induces a bilinear form
\begin{align*}
  a_s(v,w) &= ((-\LAP)^{s/2}v,(-\LAP)^{s/2}w)_{L^2(\Rd)} = \int_{\Rd} |\xi|^{2s} \calF(v)(\xi) \overline{\calF(w)(\xi)} \diff \xi \\
    &= \frac{c_{s,d}}2 \int_\Rd \int_\Rd \frac{(v(x)-v(y))(w(x)-w(y))}{|x-y|^{d+2s}} \diff y \diff x.
\end{align*}
Note that the above considerations remain  meaningful for $v,w \in \tHs$, or strictly speaking to $\widetilde v, \widetilde w \in H^s(\Rd)$, their zero extension outside $\Omega$. 
In addition, Remark~\ref{r:ths-norm} implies that $a_s$ is bounded and coercive on $\tHs$ with the convention 
$$
a_s(v,w) = ((-\LAP)^{s/2}\widetilde v,(-\LAP)^{s/2}\widetilde w)_{L^2(\Rd)}, \qquad \forall v,w \in \tHs.
$$

\subsection{The obstacle problem}
\label{sub:obsprob}

Having introduced the necessary notation we can now give a rigorous meaning to problem \eqref{eq:minprob} and study it. 
To be able to handle all the three cases under consideration (see cases~\ref{case:frac}, \ref{case:fracdrift} and \ref{case:diff} in Section~\ref{sec:intro}) in a unified way, we introduce the two--parameter space
\begin{equation}
\label{eq:defofoV}
  \Vss := \begin{dcases}
           \tHs, &\bit = 0, \\ \Hunz, &\bit =1,
         \end{dcases}
  \quad \| w \|_{\Vss}^2 := \| w \|_{\tHs}^2 + \bit \| w \|_{\Hunz}^2.
\end{equation}
From now on we assume the following assumption on the obstacle:

\begin{assumption}[obstacle]
\label{a:obstacle}
The obstacle $\chi \in C^2(\bar\Omega)$ is such that $\chi < 0$ on $\partial \Omega$. 
\end{assumption}

Under Assumption~\ref{a:obstacle} the admissible set
\begin{equation}\label{e:K}
  \calK := \left\{ w \in \Vss : w \geq \chi\ \mae \Omega \right\} \subset \Vss
\end{equation}
is nonempty, closed and convex. On $\Vss$ we define the bilinear form
\begin{equation}\label{e:calA}
  \calA(v,w) := \bit \calL(v,w) + \calD(v,w) + a_s(v, w), \qquad \forall v, w \in \Vss.
\end{equation}
Owing to Proposition~\ref{prop:drift}, it follows that $\calA$ is bounded and coercive on $\Vss$ for all cases considered.

The weak formulation of problem \eqref{eq:minprob} is defined as follows: given $f \in \Vss'$ find $u \in \calK$ such that
\begin{equation}
\label{eq:VI}
  \calA(u,u-v) \leq \scl f, u-v\scr_{\Vss',\Vss}, \quad \forall v \in \calK.
\end{equation}
Since $\calA$ is coercive, existence and uniqueness of a solution is an immediate consequence of the Lions--Stampacchia theorem \cite[Theorem II.2.1]{KinderlehrerStampacchia}. 

The next theorem guarantees the validity of the complementarity conditions \eqref{eq:minprob}.
Before proceeding, we introduce the Lagrange multiplier
\begin{equation}
\label{eq:defofLambda}
  \Lambda := \bit L u + \bbeta \cdot\GRAD u + \Laps \widetilde u -f \in \Vss'.
\end{equation}

\begin{theorem}[complementarity conditions]
\label{thm:complconds}
The solution $u \in \Vss$ of \eqref{eq:VI} satisfies
\[
  \Lambda \geq 0
\]
in $\Vss'$. In addition, if $u \in \Vss \cap C(\overline \Omega)$ then the complementarity conditions hold, \ie
\[
  \Lambda \geq 0, \qquad u \geq \chi, \qquad \Lambda(u-\chi) = 0
\]
in the sense of distributions.
\end{theorem}
\begin{proof}
Case~\ref{case:frac} is already studied in \cite[Theorem 1.2]{MNS17}; see also \cite[Proposition 2.10]{BNS}.

For Cases~\ref{case:fracdrift} and \ref{case:diff} we write \eqref{eq:VI} as
\[
   \scl \bit Lu + \bbeta \cdot \GRAD u+ \Laps \widetilde u -f, u-v\scr_{\Vss',\Vss} \leq 0, \qquad \forall v \in \calK.
\]
Let now $0 \leq \varphi \in C_0^\infty(\Omega)$ be arbitrary and set $v = u+\varphi \in \calK$ to deduce
\[
  \scl \Lambda, \varphi \scr_{\Vss',\Vss} = \scl  \bit Lu + \bbeta \cdot \GRAD u + \Laps \widetilde u - f, \varphi \scr_{\Vss',\Vss} \geq 0 .
\]
This means $\Lambda \geq 0$ in $\Vss'$ and in the sense of distributions.

In addition, if $u \in C(\overline\Omega)$, then the non-contact set
\[
  N := \left\{ x \in \Omega : u(x) > \chi(x) \right\}
\]
is open. Let $\phi \in C_0^\infty(N)$ and $\vare$ positive but sufficiently small so that $v = u \pm \vare \phi \in \calK$. This choice implies that
\[
  \scl \Lambda,\phi \scr_{\Vss',\Vss} = 0, \quad \forall \phi \in C_0^\infty(N),
\]
and the conclusion follows.
\end{proof}


\section{Regularity}\label{sec:regularity}
In this section we study the regularity of the solution to \eqref{eq:VI}. To achieve this, we first consider the linear problem without the obstacle constraint. Then, using a penalization technique, we transfer these regularity results to the solution  $u$ of \eqref{eq:VI}. In addition, using a Lewy--Stampacchia type argument, we deduce regularity properties of the Lagrange multiplier $\Lambda$ as well as the continuity of $u$, necessary to apply Theorem~\ref{thm:complconds}.

\subsection{Regularity for the linear problem}
\label{sub:reg-linear}
Here we are interested in the regularity of the solution to a linear version of \eqref{eq:VI}. Namely, given $g \in \Vss'$, we let $\Phi_g \in\Vss$ be the (unique) solution of
\begin{equation}
\label{eq:linearprob}
  \calA(\Phi_g,v) = \scl g, v \scr_{\Vss',\Vss}, \quad \forall v \in \Vss,
\end{equation}
where $\calA$ is given by \eqref{e:calA}.
We consider the regularity of each case separately.
Notice that each case requires different assumptions on the data.

\subsubsection{Case~\ref{case:frac}: Purely fractional diffusion}
Assuming $\Omega$ is of class $C^\infty$, the regularity of $\Phi_g$ was studied in \cite{MR3276603,MR0214910}. The next proposition gathers these result in our notation.

\begin{proposition}[regularity for Case~\ref{case:frac}]
\label{p:reglinear-B}
Assume that the domain $\Omega$ is of class $C^\infty$ and that, for $s\in(0,1)$, we have that $g\in H^t(\Omega)$ with $t\ge -s$. In this setting we have that $\Phi_g$, the solution of \eqref{eq:linearprob} with $\bit=0$ and $\bbeta= \mathbf 0$, satisfies
\[
	\Phi_g\in \widetilde H^{\min\{t+2s,(s+\frac12)^-\}}(\Omega), \qquad 
	\|\Phi_g\|_{\widetilde H^{\min\{t+2s,(s+\frac12)^-\}}(\Omega)}\preceq \|g\|_{H^{t}(\Omega)} .
\]
\end{proposition}

We also refer to \cite{MR3620141} for regularity results when $\Omega$ is Lipschitz and $g$ is H\"older continuous.

\subsubsection{Case~\ref{case:fracdrift}: Fractional diffusion with drift}
Recall that in this case we restrict the fractional power $s$ to $[\tfrac12,1)$. We also have $\bit=0$ and $\bbeta\neq \mathbf 0$. The proof is based on the regularity estimates for Case~\ref{case:frac} presented in Proposition~\ref{p:reglinear-B} and techniques developed in \cite{PengXXXX}.

\begin{proposition}[regularity for Case~\ref{case:fracdrift}]\label{l:reglinear-C}
Assume that the domain $\Omega$ is of class $C^\infty$ and that $g\in L^2(\Omega)$. Let $\Phi_g$ be the solution of \eqref{eq:linearprob} with $s \in [\tfrac12,1)$, $\bit=0$ and $\bbeta\neq \mathbf 0$.
\begin{enumerate}[a)]
\item \label{item:sge12} If $s>\tfrac12$, then $\Phi_g\in \widetilde H^{(s+\frac12)^-}(\Omega)$ and satisfies 
\[
	\|\Phi_g\|_{\widetilde H^{(s+\frac12)^-}(\Omega)} \preceq \|g\|_{L^2(\Omega)} .
\]

\item \label{item:seq12} If $s=\tfrac12$, there exists a positive constant $C_{\frac12}$ such that when $\|\bbeta\|_{L^\infty(\Omega,\Real^d)}< C_{\frac12}$, we have that $\Phi_g\in \widetilde H^{1^-}(\Omega)$ with the corresponding estimate. Otherwise, that is when $\|\bbeta\|_{L^\infty(\Omega,\Real^d)}\ge C_{\frac12}$, then there exists $\delta\in (0,\tfrac12)$ such that $\Phi_g\in \widetilde H^{\frac12+\delta}(\Omega)$ with the corresponding estimate.
\end{enumerate}
\end{proposition}
\begin{proof}
We consider each case separately. 

We begin the treatment of Case~\ref{item:sge12} by rewriting the linear problem as follows: find $\Phi_g\in \tHs$ satisfying
\[
	a_s(\Phi_g, v) = (g,v)_{L^2(\Omega)}-\scl \bbeta \cdot \nabla \Phi_g,v\scr_{H^{-s}(\Omega),\tHs} =: \scl G,v\scr_{H^{-s}(\Omega),\tHs}, \ \forall v\in \tHs.
\]
Now, using a bootstrapping argument, we improve the regularity of $\Phi_g$. 
Starting from $\Phi_g\in \tHs$, we first notice that, according to Proposition~\ref{prop:drift}, $G\in H^{s-1}(\Omega)$. Thanks to Proposition~\ref{p:reglinear-B} with $t=s-1$ we get $\Phi_g\in \widetilde H^{\min\{3s-1,(s+\frac12)^-\}}(\Omega)$.
Invoking Propositions~\ref{prop:drift} and~\ref{p:reglinear-B} again, we deduce that
\[
	\Phi_g\in \widetilde H^{\min\{5s-2,(3s-\frac12)^-,(s+\frac12)^-\}}(\Omega) = \widetilde H^{\min\{5s-2,(s+\frac12)^-\}}(\Omega).
\]
Repeating the above argument $n$ times, we arrive at
\[
	\Phi_g\in \widetilde H^{\min\{(2n+1)s-n,(s+\frac12)^-\}}(\Omega).
\] 
From the assumption $s>\frac 1 2$, we have $(2n+1)s-n\to \infty$ as $n\to \infty$ so that setting $n=\lceil\tfrac1{4s-2}\rceil$ yields the desired result for case~\ref{item:sge12}, \ie $\Phi_g \in \widetilde H^{(s+\frac12)^-}(\Omega)$.

Let us now show Case~\ref{item:seq12} using a perturbation argument. Denote by $T: \widetilde H^{\frac12}(\Omega)\to H^{-\frac12}(\Omega)$ the unbounded operator satisfying
\[
	\scl T g, v \scr_{H^{-\frac12}(\Omega),\widetilde H^{\frac12}(\Omega)} 
	= c_{\frac12}~a_{\frac12}(g,v), \quad \forall v\in \widetilde H^{\frac12}(\Omega) ,
\]
where $c_{\frac12}$ denotes the normalization constant such that 
\[
	\|w\|_{\dot H^{\frac12}(\Omega)}^2=c_{\frac12} \|w\|_{\widetilde H^{\frac12}(\Omega)}^2.
\]
As we shall see, the purpose of the normalization by $c_{1/2}$ is to relate the functional spaces $\widetilde H^{r}(\Omega)$ to the interpolation spaces $\dot H^r(\Omega)$ and invoke  operator interpolation results.
Proposition~\ref{p:reglinear-B} guarantees that the inverse of $T$ is a bounded operator mapping $H^{t}(\Omega)$ to $\widetilde H^{\min\{t+1,1^-\}}(\Omega)$ with $t\ge -\tfrac12$. Given $\eta \in (0,1]$, we rewrite the linear problem \eqref{eq:linearprob} in the form of a perturbation of the identity 
\[
	\scl (I-B)\Phi_g , v \scr_{H^{-\frac12}(\Omega),\widetilde H^{\frac12}(\Omega)}  
	=  \scl \eta c_{\frac12} T^{-1} g ,v \scr_{H^{-\frac12}(\Omega),\widetilde H^{\frac12}(\Omega)} , \quad \forall v\in \tHs ,
\]
where $B:=(1-\eta)I - \eta c_{\frac12} T^{-1}\bbeta \cdot \nabla$. We next investigate the mapping properties of the operator $B$ using the equivalent interpolation norm $\dot H^t(\Omega)$ with $t\in[\tfrac12,1)$. For $w \in \dot H^{1^-}(\Omega)$, we have
\begin{equation}\label{i:B-bound}
\begin{aligned}
	\|Bw \|_{\dot H^{1^-}(\Omega)} & \le (1-\eta) \|w \|_{\dot H^{1^-}(\Omega)} 
	+ \eta\|c_{\frac12}T^{-1}\bbeta \cdot \nabla w \|_{\dot H^{1^-}(\Omega)} \\
	&\le  ((1-\eta) +C\eta\|\bbeta\|_{L^\infty(\Omega,\Real^d)})\|w \|_{\dot H^{1^-}(\Omega)}=: M_1(\eta)\|w \|_{\dot H^{1^-}(\Omega)} . \\
\end{aligned}
\end{equation}
Here the constant $C$ depends on the constants in the estimates of Proposition~\ref{prop:drift}, Proposition~\ref{p:reglinear-B} and $c_{\frac12}$. Setting $C_{\frac12} := 1/C$, the condition $\|\bbeta\|_{L^\infty(\Omega,\Real^d)}<C_{\frac12}$ guarantees that $M_1(\eta)<1$ for any $\eta\in(0,1]$. In turn, this implies that $I-B: \dot H^{1^-}(\Omega) \rightarrow \dot H^{1^-}(\Omega)$ is invertible and
\[
\begin{aligned}
	&\|(I-B)^{-1}\|_{\widetilde H^{1^-}(\Omega) \to \widetilde H^{1^-}(\Omega)} \\
	& \qquad \preceq \|(I-B)^{-1}\|_{\dot H^{1^-}(\Omega) \to \dot H^{1^-}(\Omega)} 
	&\le \sum_{j=0}^\infty \|B\|_{\dot H^{1^-}(\Omega) \to \dot H^{1^-}(\Omega)}^j 
	\le \frac1{1-M_1(\eta)} .
\end{aligned}
\] 
Hence we deduce that $\Phi_g\in \widetilde H^{1^-}(\Omega)$ and
\[
	\|\Phi_g\|_{\widetilde H^{1^-}(\Omega)} \le \eta\|(I-B)^{-1}\|_{\widetilde H^{1^-}(\Omega) \to \widetilde H^{1^-}(\Omega)} 
	 \|c_{\frac12}T^{-1}g\|_{\widetilde H^{1^-}(\Omega)} \preceq \|g\|_{L^2(\Omega)} .
\]
Instead, when $\|\bbeta\|_{L^\infty(\Omega,\Real^d)}\ge C_{\frac12}$, we note that for $ w \in \dot H^{\frac12}(\Omega)$,
\begin{align*}
	\|B w \|_{\dot H^{\frac12}(\Omega)}^2 &= c_{\frac12}\|B w \|_{\widetilde H^{\frac12}(\Omega)}^2 
    =\scl T B w ,B w  \scr_{H^{-\frac12}(\Omega),\widetilde{H}^{\frac12}(\Omega)} \\
	&= (1-\eta)^2\scl T w , w \scr_{H^{-\frac12}(\Omega),\widetilde{H}^{\frac12}(\Omega)} \\
	&-(1-\eta)\eta c_{\frac12} \left[ 
    \scl T w , T^{-1}\bbeta \cdot \nabla w \scr_{H^{-\frac12}(\Omega),\widetilde{H}^{\frac12}(\Omega)} \right. \\
    &+ \left. \scl T T^{-1}\bbeta \cdot \nabla w , w \scr_{H^{-\frac12}(\Omega),\widetilde{H}^{\frac12}(\Omega)} \right]\\
	&+\eta^2 c_{\frac12}^2 \scl T T^{-1}\bbeta \cdot \nabla w , T^{-1}\bbeta \cdot \nabla w \scr_{H^{-\frac12}(\Omega),\widetilde{H}^{\frac12}(\Omega)}\\
	& = (1-\eta)^2 \| w \|_{\dot H^{\frac12}(\Omega)}^2 + \eta^2 |c_{\frac12}^2 \| T^{-1}\bbeta \cdot \nabla w \|_{\dot H^{\frac12}(\Omega)}^2\\
	& \le (1-\eta)^2 \| w \|_{\dot H^{\frac12}(\Omega)}^2 +\widetilde C \eta^2\|\bbeta\|_{L^\infty(\Omega,\Real^d)}^2 \| w \|_{\dot H^{\frac12}(\Omega)}^2 ,
\end{align*}
where in the third equality we used the symmetry of $T$ and \eqref{e:beta_div_free_cons}. The positive constant $\widetilde C$ depends on the same parameters as $C_{1/2}$. The optimal choice for $\eta$ is $\eta^* := 1/(1+\widetilde C\|\bbeta\|_{L^\infty(\Omega,\Real^d)}^2)\in (0,1)$, which leads to
\begin{equation}\label{i:B-bound-2}
	\|B w \|_{\dot H^{\frac12}(\Omega)} \le \sqrt{1-\frac1{1+\widetilde C\|\bbeta\|_{L^\infty(\Omega,\Real^d)}^2}}\| w \|_{\dot H^{\frac12}(\Omega)}=:M_2\| w \|_{\dot H^{\frac12}(\Omega)},
\end{equation}
with $M_2<1$.
From  \eqref{i:B-bound} and \eqref{i:B-bound-2}, we obtain by interpolation
\[
	\|B w \|_{\dot H^{(\frac12+r)^-}(\Omega)}
	\le M_1(\eta^*)^{2r} M_2^{1-2r} \| w \|_{\dot H^{(\frac12+r)^-}(\Omega)},\quad\text{for }r\in (0,\tfrac12) .
\]
and upon selecting $r>0$ sufficiently small so that
\[
	M_1^{2r}(\eta^*) M_2^{1-2r}<1,
\]
we obtain that $B$ is a bounded operator in $\dot H^{(\frac12+r)^-}(\Omega)$ and so 
$$
\Phi_g \in  \widetilde H^{\frac12+\delta}(\Omega)
$$
for some $\delta \in (0,\frac 1 2)$ as asserted.
\end{proof}

\subsubsection{Case~\ref{case:diff}: Integro--differential operator}

We let $\bit=1$ and immediately notice that $\calV_{s,1} = \Hunz$ for all values of $s$. Our results rely on the following regularity assumption for a second order elliptic problem.

\begin{assumption}[elliptic regularity]
\label{a:elliptic-regularity}
Let $g\in H^{-1}(\Omega)$, and $w_g\in H^1_0(\Omega)$ be the unique solution of
\begin{equation}\label{e:without-frac}
	\calL(w_g,v)  = \scl g,v\scr_{H^{-1}(\Omega),\Hunz}, \quad \forall v\in \Hunz.
\end{equation}
There exists $r\in(0,1]$ and a constant constant $C_r$ so that
\[
	\|w_g\|_{H^{1+r}(\Omega)}\leq C_r \|g\|_{H^{-1+r}(\Omega)}.
\]
In particular, we have
\[
	\|w_g\|_{H^{1+\gamma}(\Omega)}\leq C_r \|g\|_{H^{-1+\gamma}(\Omega)}
\]
for all $\gamma \in (0,r]$.
\end{assumption}

We note that $r$ and $C_r$ depend on the smoothness of the domain $\Omega$ and the coefficients $A$ and $c$. For example, if $\Omega$ is a polytope and the bilinear form $\calL$ is the Dirichlet form, \ie
\begin{equation}\label{e:dirichlet-form}
	\calL(v,w) = \int_\Omega \nabla v \cdot \nabla w \diff x, \quad \forall v,w\in H^1_0(\Omega),
\end{equation}
then, according to \cite{dauge}, Assumption~\ref{a:elliptic-regularity} holds for some $ \frac 1 2 < r \leq 1$ that depends on the shape of the domain.

To concisely state the regularity result obtained in this case, we define
\begin{equation}\label{eq:defofmu}
  \mu := \mu(s,r) := \begin{dcases}
                                 1+ r  & 0<s<\tfrac54-\tfrac r 2, \\
                                 (\tfrac72-2s)^-,  & \tfrac54-\tfrac r 2\le s <1.
                   \end{dcases}
\end{equation}

\begin{proposition}[regularity for Case~\ref{case:diff}]
\label{prop:reglinearidiff}
Let $\Phi_g$ be the solution to problem \eqref{eq:linearprob} with $\bit=1$. Let $ r \in(0,1]$ be the regularity index given in Assumption~\ref{a:elliptic-regularity} and $g \in \Ldeux$. Then we have that 
\[
  \Phi_g \in  H^\mu(\Omega) \cap H^1_0(\Omega), \qquad 
  \| \Phi_g \|_{H^\mu(\Omega)} \preceq \| g \|_\Ldeux,
\]
where the hidden constant depends on $\Omega$, $C_r$ in Assumption~\ref{a:elliptic-regularity} and $\bbeta$. 
\end{proposition}
\begin{proof}
Notice that the unique solution $\Phi_g \in \Hunz$ of problem \eqref{eq:linearprob} is also the unique solution of
\begin{equation}
\label{eq:modifiedlinear}
  \calL(\Phi_g, v) = (g-\bbeta \cdot \GRAD\Phi_g ,v)_\Ldeux - \scl \Laps \widetilde \Phi_g, v \scr_{\Hmun,\Hunz}=: \scl G, v \scr_{\Hmun,\Hunz}.
\end{equation}

We discuss the case $r\ge\tfrac12$ and split the proof in several (sub-)cases.
\begin{enumerate}[$\bullet$]
  \item \emph{Case $s \in (0,\tfrac{2- r }2]$}: According to \cite[Theorem XI.2.5]{TaylorME} we have that $\Laps \widetilde \Phi_g \in H^{1-2s}(\Omega) \subset H^{-1+ r }(\Omega)$. From the elliptic regularity assumption we conclude that $\Phi_g \in H^{1+ r }(\Omega)=H^\mu(\Omega)$ with the corresponding estimate.
  
  \item \emph{Case $s \in (\tfrac{2- r }2,\tfrac34]$}: If this is the case, we now have that $G \in H^{1-2s}(\Omega)$ so that invoking the elliptic regularity assumption again with $\gamma=2-2s<r$ (see Assumption~\ref{a:elliptic-regularity}) and using the norm equivalence property described in Remark~\ref{r:norm-equivalence}, we obtain that $\Phi_g \in H^{3-2s}(\Omega)\cap\Hunz$. 
 Because $s\in (\tfrac{2- r }2,\tfrac34]$, we can only conclude that $\Phi_g \in \widetilde{H}^{\tfrac32-\epsilon}(\Omega)$.
 However, we can repeat the process because in that case $\Laps \widetilde \Phi_g \in H^{(\frac32-2s)^-}(\Omega)$ and thus $G \in H^{(\frac32-2s)^-}(\Omega)$. 
 From the elliptic regularity assumption we obtain that $\Phi_g \in H^{\min\{(\frac72-2s)^-,1+ r \}}(\Omega) = H^\mu(\Omega)$ with the corresponding estimate.
  
  \item \emph{Case $s \in (\tfrac34,\tfrac78]$}: Proceeding as the previous case, we have that $\Phi_g\in H^{3-2s}(\Omega)\cap H^1_0(\Omega)=\widetilde H^{3-2s}(\Omega)$ and thus $G \in H^{3-4s}(\Omega)$. 
  The elliptic regularity assumption this time with $\gamma = \min\{r,4-2s\}$ yields $\Phi_g \in H^{\min\{5-4s,1+ r \}}(\Omega)\cap H^1_0(\Omega)\subset \widetilde{H}^{(\frac32)^-}(\Omega)$. Continuing further, we have $G \in H^{(\frac32-2s)^-}(\Omega)$, and finally that $\Phi_g \in H^{\min\{(\frac72-2s)^-,1+ r \}}(\Omega) = H^\mu(\Omega)$, with the corresponding estimate.

  \item \emph{General case, $s \in (\tfrac{4n-1}{4n},\tfrac{4n+3}{4n+4}]$ for $n\geq2$}: We proceed as in the previous steps to obtain after a finite number of iterations that $\Phi_g \in H^{\min\{(\frac72-2s)^-,1+ r \}}(\Omega) = H^\mu(\Omega)$ with the corresponding estimate.
\end{enumerate}

The proof for the case $r<\frac 1 2$ is omitted for brevity as it follows invoking similar arguments but decomposing $(0,1)$ as $(0,\tfrac{2- r }2]$, $(\tfrac{2- r }2, \tfrac{4- r }4]$
and $\cup_{n\geq 2} (\tfrac{2n- r }{2n},\tfrac{2n+2- r }{2n+2}]$.
\end{proof}

\begin{remark}[polygonal domains in $\Real^2$]\label{r:poly}
Let us consider two special cases in $\Real^2$. If $\Omega$ is a convex polygon in $\Real^2$, the coefficient matrix $A$ and zero order term $c$ are smooth enough, then we obtain full elliptic regularity for problem \eqref{e:without-frac}, \ie $ r =1$. In this case, according to Proposition~\ref{prop:reglinearidiff}, 
\[
	\Phi_g\in H^1_0(\Omega)\cap
	\begin{dcases}
                                 H^2(\Omega) & 0<s<\tfrac34, \\
                                 H^{(\tfrac72-2s)^-}(\Omega),  &\tfrac34\le s <1.
        \end{dcases}
\]
If, on the other hand, $\Omega$ is a L--shaped domain (e.g. $(-1,1)^2\setminus [0,1]^2$), $\bbeta=\mathbf 0$ and $\calL(\cdot,\cdot)$ is the Dirichlet form, then we have that $ r =\tfrac23$ and hence
\[
	\Phi_g\in H^1_0(\Omega)\cap
	\begin{dcases}
                                 H^{\tfrac53}(\Omega) & 0<s<\tfrac{11}{12}, \\
                                 H^{(\tfrac72-2s)^-}(\Omega),  &\tfrac{11}{12}\le s <1.
        \end{dcases}
\]
\end{remark}

\begin{remark}[continuity of $\Phi_g$]\label{r:continuity}
If the elliptic regularity index in Assumption~\ref{a:elliptic-regularity} is above $\frac 12$, then by $\Phi_g \in C(\overline{\Omega})$ by Sobolev embedding.
\end{remark}


\subsection{Regularity of the obstacle problem}
\label{sub:regobstacle}

The regularity estimates for the linear problem are instrumental to obtain regularity properties of the solution to the obstacle problem \eqref{eq:VI}.
To achieve this, we follow the penalization ideas from \cite[Section IV.2]{KinderlehrerStampacchia}, see also \cite{MNS17,MR3090147}. 
We begin by recalling that Assumption~\ref{a:obstacle} guarantees, at least heuristically, that the contact set is separated from the boundary of the domain $\partial\Omega$.
Of particular importance below is that, once again owing to Assumption~\ref{a:obstacle}, it is possible to extend $\chi$ to a larger domain: we denote by $\calW \subset \Rd$ a domain with smooth boundary such that $\Omega \Subset \calW$ and by $\calE \chi \in C^2_0(\calW)$ an extension of $\chi$ to $\calW$
\[
  \calE\chi_{|\Omega} = \chi, \qquad \calE\chi_{|\complement\Omega} \leq 0.
\]
The choice of $\calW$ and $\calE\chi$ is arbitrary but irrelevant for the results derived below.

Next, we assume certain regularity and compatibility between the operator and problem data.

\begin{assumption}[smoothness and data compatiblity]
\label{a:mainassumption_regularity}
Let $p > \max\{2, d/(2s) \}$ in Cases~\ref{case:frac} and \ref{case:fracdrift}, and $p \geq 2$ in Case~\ref{case:diff}.
We have that $f \in L^p(\Omega)$ and
\begin{equation}\label{e:def_F}
  F:= \bit L \chi +\bbeta \cdot \nabla\chi +(-\Delta)^s \widetilde{\calE\chi}-f
 \end{equation}
 is an absolutely continuous measure with respect to the Lebesgue measure and its Radon--Nikodym derivative belongs to $L^p(\Omega)$.
\end{assumption}

Furthermore, we gather the assumptions required for the regularity results obtained in Section~\ref{sub:reg-linear} on the linear problem in the next assumption.

\begin{assumption}[regularity of the linear problem]
\label{a:regLinear}
The regularity results of the linear problem \eqref{eq:linearprob} obtained in Section~\ref{sub:reg-linear} are valid, \ie we assume that
\begin{enumerate}[$\bullet$]
  \item Case~\ref{case:frac}: The domain $\Omega$ is of class $C^\infty$.
  
  \item Case~\ref{case:fracdrift}: The domain $\Omega$ is of class $C^\infty$.
  If $s=\frac12$, the drift magnitude is sufficiently small, \ie $\|\bbeta \|_{L^\infty(\Omega,\Real^d)} < C_{\frac12}$, where $C_{\frac12}$ is the constant in Proposition~\ref{l:reglinear-C}.
  
  \item Case~\ref{case:diff}: The elliptic regularity assumption (Assumption~\ref{a:elliptic-regularity}) holds for an index $r \in (\frac12,1]$.
\end{enumerate}
\end{assumption}

\remark[Case~\ref{case:fracdrift}]{
To simplify the discussion, we do not discuss the case when $\|\bbeta\|_{L^\infty(\Omega,\Real^d)}\ge C_{\tfrac12}$.
However, the argumentation below extends to this case in view of the regularity property obtained in Proposition~\ref{l:reglinear-C}. 
}

We prove below that the solution to the obstacle problem belongs to $H^\sigma(\Omega)$, where
\begin{equation}
\label{eq:defofsigma}
  \sigma := \sigma(\bit,\bbeta,s, r) = \begin{dcases}
                              \min\left\{2s,\left(s+\frac12\right)^-\right\}, & \textrm{Case}~\ref{case:frac}, \\
                              \left(s+\frac12\right)^-, & \textrm{Case}~\ref{case:fracdrift}, \\
                              \mu(s,r), & \textrm{Case}~\ref{case:diff},
                            \end{dcases}
\end{equation}
where $\mu(s,r)$ is defined by \eqref{eq:defofmu}.
The first step is to analyze a penalization problem.

\subsubsection{Penalization}
\label{sub:penalty}

Given $\vare>0$, let $\vartheta_\vare \in C^\infty(\Real)$ be such that $|\vartheta_\vare|\le 1$, it is non increasing and
\[
	\vartheta_\vare(t) := 
	\begin{dcases}
		1, & t\le0,\\
		0, & t\ge\vare.
	\end{dcases}
\]

Under Assumption~\ref{a:obstacle} and for $f \in \Ldeux$, the penalized problem constructs an approximation of $u$ by $u_\vare\in \Vss$ defined as the solution to 
\begin{equation}\label{e:penalized-problem}
	\calA(u_\vare,v) = (\max\{F,0\} \vartheta_\vare(u_\vare-\chi) +f,v)_{L^2(\Omega)},
	\quad \forall v\in \Vss,
\end{equation}
where $\calA$ is given by \eqref{e:calA}; compare with \eqref{eq:minprob}. 
Notice that \eqref{e:penalized-problem} is a variational problem with a strictly coercive and monotone operator on $\Vss$ and therefore has a unique solution. The next lemma gathers properties of the penalized solution.

\begin{lemma}[two--sided uniform bounds]\label{l:pointwise}
Suppose that Assumption~\ref{a:obstacle} holds and $f \in \Ldeux$. Let $u, u_\vare \in \Vss$ be the solutions to \eqref{eq:VI} and \eqref{e:penalized-problem}, respectively. Then we have that
\[
	u\le u_\vare\le u+\vare,\quad \text{a.e. in } \Omega.
\]
\end{lemma}
\begin{proof}
We start by noting that for $w\in \Vss$
$$
\calL(\max\{w,0\},\max\{w,0\}) + \calD(\max\{w,0\},\max\{w,0\}) =  \calL(w,\max\{w,0\}) + \calD(w,\max\{w,0\}). 
$$
Owing to \cite[Lemma 2.1, iii)]{MNS17}, this property also holds for $a_s$, \ie
\[
  a_s(\max\{w,0\},\max\{w,0\}) \le a_s(w,\max\{w,0\}).
\]
This together with the coercivity of $\calA$ yield
\begin{align*}
	\|\max\{\chi-u_\vare,0\}\|^2_{\Vss} 
	 \le & \bit \calL(\chi-u_\vare,\max\{\chi-u_\vare,0\}) + a_s (\calE \chi - u_\vare, \max\{{\calE\chi}- u_\vare,0\} ) \\
	&+ \calD (\chi-u_\vare,\max\{\chi-u_\vare,0\}).
\end{align*}
Hence, the definition  \eqref{e:def_F} of $F$ and the relation \eqref{e:penalized-problem} satisfied by $u_\vare$ imply that 
\begin{equation*}
\begin{split}
\|\max\{\chi-u_\vare,0\}\|^2_{\Vss} & \preceq	  (F,\max\{\chi-u_\vare,0\})_{L^2(\Omega)}  \\
&\qquad  - (\max\{F,0\}\vartheta_\vare(u_\vare-\chi),\max\{\chi-u_\vare,0\})_{L^2(\Omega)} \\
&\preceq(\max\{F,0\}(1-\vartheta_\vare(u_\vare-\chi)),\max\{\chi-u_\vare,0\})_{L^2(\Omega)}.
\end{split}
\end{equation*}
Observing that $\vartheta_\vare(u_\vare-\chi) =1$ whenever $\chi -u_\vare \geq 0$, we deduce that $\|\max\{\chi-u_\vare,0\}\|_{\Vss}=0$ and in particular $u_\vare \geq \chi$ a.e. in $\Omega$.
In other words, we have that $u_\vare\in \calK$. 
Since $\vartheta_\vare \geq 0$, $\max\{F,0\}\vartheta_\vare + f\geq f$ and therefore $u_\vare$ is a supersolution to problem \eqref{eq:VI} (\cf \cite[Definition 5.6]{KinderlehrerStampacchia}). Following the argumentation in the proof of Theorem~6.4 in \cite{KinderlehrerStampacchia}, we then obtain that $u_\vare\ge u$. This proves the first claimed inequality.

For the second inequality, we proceed similarly but invoking part iii of Lemma 2.3 in \cite{MNS17} instead of part iii of Lemma 2.1 to write
\begin{align*}
	& \|\max\{u_\vare-u-\vare,0\}\|_{\Vss}^2 \\
&\qquad \preceq  \bit \calL(u_\vare-u,\max\{u_\vare-u-\vare,0\}) + a_s (u_\vare- u , \max\{u_\vare- u-\vare,0\}) \\
	& \qquad  \qquad + \calD (u_\vare-u,\max\{u_\vare-u-\vare,0\}) \\
	& \qquad \preceq \bit \calL(u_\vare,\max\{u_\vare-u-\vare,0\}) + a_s (u_\vare, \max\{u_\vare-\vare,0\}) \\
	& \qquad  \qquad + \calD (u_\vare,\max\{u_\vare-u-\vare,0\}) - (f,\max\{u_\vare-u-\vare,0\})\\
	& \qquad \preceq (\max\{F,0\} \vartheta_\vare(u_\vare-\chi) ,\max\{u_\vare-u-\vare,0\})_{L^2(\Omega)}
	=0 .
\end{align*}
Therefore, we have $u_\vare\le u+\vare$ a.e. in $\Omega$.
This completes the proof.
\end{proof}

We are now in position to derive the main result on the regularity of the solution to the obstacle problem.

\begin{theorem}[regularity of $u$]
\label{thm:regularityVI}
Suppose that Assumptions~\ref{a:obstacle},~\ref{a:mainassumption_regularity} and \ref{a:regLinear} hold. Then the solution $u \in \Vss$ of the obstacle problem \eqref{eq:VI} satisfies $u \in H^\sigma(\Omega)$, where  $\sigma$ is given by \eqref{eq:defofsigma}.
Moreover, we have
\[
  \|u\|_{H^\sigma(\Omega)} \preceq \|f\|_{L^2(\Omega)} + \|\max\{F,0\}\|_{L^2(\Omega)}.
\]
\end{theorem}
\begin{proof}
It suffices to observe that under Assumptions~\ref{a:obstacle}, and \ref{a:mainassumption_regularity}, the right--hand side of the penalized problem \eqref{e:penalized-problem} belongs to $L^2(\Omega)$.
Whence, the conditions necessary to invoke Assumption~\ref{a:regLinear} are fulfilled, and the regularity results of Section~\ref{sub:reg-linear} imply that
$u_\vare \in H^\sigma(\Omega)$ and that the following estimate holds
\[
	\|u_\vare\|_{H^\sigma(\Omega)} \preceq \|f\|_{L^2(\Omega)} + \|\max\{F,0\}\|_{L^2(\Omega)} .
\]
Hence, there exists a sequence $\{ u_{\vare_j} \}_{j\geq 0}$ with $\vare_j \to 0$ when $j\to \infty$ and $\overline u \in H^\sigma(\Omega)$ such that $u_{\vare_j}$ converges weakly to $\overline u$ in $H^\sigma(\Omega)$. 
The compact embedding of $H^\sigma(\Omega)$ into $L^2(\Omega)$ guarantees that $u_{\vare_j}$ (up to a not relabeled subsequence)  strongly converges to $\overline u$ in $L^2(\Omega)$. 
According to Lemma~\ref{l:pointwise}, we also have that $u_{\vare_j}$ converges to $u$ almost everywhere and so $u=\overline u$ almost everywhere thanks to the Lebesgue dominated convergence theorem.
This completes the proof.
\end{proof}

\begin{remark}[another penalization]
\label{rm:otherpenalty}
Notice that, at least in Case~\ref{case:diff}, we could have used the penalization technique detailed in \cite[Section IV.5]{KinderlehrerStampacchia}. This would allow for the more general differential operator $L$ with suitable monotonicity and coercivity properties.
\end{remark}

\subsubsection{Regularity of $\Lambda$ and continuity of $u$}
\label{ssub:regLambda}

The numerical  approximation of the obstacle problem proposed bellow requires (i) further regularity of the Lagrange multiplier $\Lambda$, defined in \eqref{eq:defofLambda} and (ii) the validity of the complementary conditions  \eqref{eq:minprob}.
In view of Theorem~\ref{thm:complconds}, the later requires the continuity of the solution to the obstacle problem.
This section is devoted to these two properties.

Let us begin by showing the regularity of the Lagrange multiplier. We propose a modification of Theorem 4.2.1 in \cite{MR880369} 
and emphasize that the latter cannot be directly applied.
Indeed, the abstract Theorem 4.2.1 in \cite{MR880369} requires that  $\calA(\chi,\cdot) \in \Vss'$, which is not meaningful in our context (we can only apply $\Laps$ to an extension $\calE \chi$).

\begin{lemma}[Lewy-Stampacchia type estimate]\label{l:LS}
Under Assumptions~\ref{a:obstacle}, \ref{a:mainassumption_regularity} and \ref{a:regLinear} then 
\[
  \scl \Lambda ,\phi\scr_{\Vss',\Vss} \le (\max\{F,0\},\phi)_{L^2(\Omega)}, \quad\forall \phi\in C_0^\infty(\Omega),\phi\ge 0 .
\]
\end{lemma}
\begin{proof}
As in the proof of Theorem~\ref{thm:regularityVI}, we construct a subsequence $u_{\vare_j}$ strongly converging in $L^2(\Omega)$ to $u$.
Hence, for all non-negative $\phi\in C_0^\infty(\Omega)$, we have
\begin{align*}
\scl \Lambda ,\phi\scr_{\Vss',\Vss} &= 
  \int_\Omega \left[ u\left(\bit L \phi + (-\Delta)^s\widetilde\phi -\bbeta \cdot \nabla\phi\right)  -f\phi\right] \diff x\\
	&=\lim_{j \to \infty}  \int_\Omega \left[ u_{\vare_j} \left(\bit L \phi + (-\Delta)^s\widetilde\phi -\bbeta \cdot \nabla\phi\right)  -f\phi\right] \diff x \\
	&=\lim_{j \to \infty} \int_\Omega\max\{F,0\}\vartheta_{\vare_j}(u_{\vare_j}-\chi)\phi\,\diff x
	\le \left(\max\{F,0\},\phi\right)_{L^2(\Omega)},
\end{align*}
as claimed.
\end{proof}

We can now derive the regularity of the Lagrange multiplier. 
The proof follows from Lemma~\ref{l:LS} and, essentially, repeats the arguments given in \cite[Theorem 4.2.4]{MR880369}.

\begin{theorem}[regularity of $\Lambda$]
\label{thm:Lambdaregular}
Suppose that Assumptions~\ref{a:obstacle}, \ref{a:mainassumption_regularity} and \ref{a:regLinear} hold. Then we have that the Lagrange multiplier $\Lambda$, defined in \eqref{eq:defofLambda}, satisfies
\[
  \Lambda \in L^p(\Omega), \qquad \| \Lambda \|_{L^p(\Omega)} \leq \| \max\{F,0\} \|_{L^p(\Omega)}.
\]
\end{theorem}
\begin{proof}
Since $\Lambda \in \Vss'$ and $\Lambda \geq 0$ in the sense of distributions (Theorem~\ref{thm:complconds}), it follows from the Riesz-Schwartz theorem (see \cite[Th\'eor\`eme I.4.V]{MR0209834} and \cite[Theorem 1.7.II]{MR2012831}) that $\Lambda$ is a positive Radon measure. The Lewy-Stampacchia estimate of Lemma~\ref{l:LS} then implies that this measure is absolutely continuous with respect to the Lebesgue measure and that its Radon-Nikodym derivative belongs to $L^p(\Omega)$ (thanks to Assumption~\ref{a:mainassumption_regularity}) with the asserted estimate.
\end{proof}

From the above result, we deduce the continuity of the solution and, as a consequence, that the assumptions of Theorem~\ref{thm:complconds} are satisfied.

\begin{theorem}[continuity of $u$]
\label{t:continuity-C}
Suppose that Assumptions~\ref{a:obstacle}, \ref{a:mainassumption_regularity} and \ref{a:regLinear} hold. Assume in addition that, for Case~\ref{case:fracdrift}, we have that $s \in (\tfrac{d+1}6,1)\cap (\tfrac12,1)$. The solution $u \in \Vss$ to the obstacle problem \eqref{eq:VI} has a continuous representative in its class of equivalence.
\end{theorem}
\begin{proof}
We consider each case separately:
\begin{enumerate}[$\bullet$]
  \item Case~\ref{case:frac}: Since $f \in L^p(\Omega)$ with $p>d/(2s)$ the continuity follows from \cite[Theorem 1.2]{MNS17}.
  
  \item Case~\ref{case:fracdrift}: We have that $u \in \widetilde H^{(s+1/2)^-}(\Omega)$ and thus
  \[
    \bbeta \cdot \nabla u\in H^{(s-\frac 12)^-}(\Omega)\subset L^{q}(\Omega), \qquad q:=\frac{2d}{(d+1-2s+2\epsilon)},
  \]
  for every $\epsilon>0$. 
  From Theorem~\ref{thm:Lambdaregular} we have that $\Lambda \in L^p(\Omega)$ and so
  \[
    \Laps \tilde u = \Lambda +f - \bbeta \cdot \GRAD u \in L^{\min\{p,q\}}(\Omega).
  \]
 We also use the assumption $s>\tfrac{d+1}6$ to deduce that $q>d/(2s)$ provided $\epsilon$ is chosen sufficiently small.
 Therefore, Proposition 1.4 in \cite{MR3216831} guarantees that $u$ is continuous.

  \item Case~\ref{case:diff}: Because $u \in H^\mu(\Omega)$, its continuity directly follows by Sobolev embedding, see Remark~\ref{r:continuity}.
  
\end{enumerate}
This ends the proof.
\end{proof}


\section{Finite element approximation}
\label{sec:FEM}

Having studied problem \eqref{eq:VI}, its properties and the regularity of its solutions, we can now present a discrete counterpart along with its analysis. We begin by assuming without loss of generality that $\overline\Omega$ is contained in the unit ball of $\Rd$. Let $\{ \calT_h(\Omega) \}_{h>0}$ be a family of conforming simplicial triangulations of $\overline\Omega$. We assume that these triangulations are shape-regular and quasi-uniform in the sense of \cite{MR1930132,ERNGuermond} and identify $h$ with the maximal simplex size. 

Over $\calT_h(\Omega)$ we construct $\polV_h$, the space of piecewise affine functions subordinate to $\calT_h(\Omega)$ that vanish on $\partial\Omega$. An instrumental tool for the analysis that we shall perform is the use of $I_h$, the positivity preserving interpolant introduced in \cite[Section 3]{ChenNochetto}. For convenience we recall some of its basic properties and establish a stability estimate for it in fractional Sobolev spaces of order $\beta \in (0,\tfrac32)$. 

\begin{proposition}[properties of $I_h$]
\label{prop:properitesIh}
Let $I_h : L^1(\Omega) \to \polV_h$ be the positivity preserving interpolation operator of \cite{ChenNochetto}. This operator satisfies:
\begin{enumerate}[1.]
  \item \emph{Positivity}: If $w \geq 0$ \mae in $\Omega$, then $I_h w \geq 0$.

  \item \emph{$L^2(\Omega)$--approximation}: If $w \in \Hunz \cap H^\beta(\Omega)$ with $\beta \in [1,2]$, then
  \[
    \| w - I_h w \|_\Ldeux \preceq h^\beta \| w \|_{H^\beta(\Omega)}.
  \]
  
  \item \emph{$\Vss$--approximation}: If $w \in \Hunz \cap H^\beta(\Omega)$ with $\beta \in [1,2]$, then
  \[
    \| w - I_h w \|_{\widetilde H^s(\Omega)} \preceq h^{\beta-s} \| w \|_{H^\beta(\Omega)},  \qquad \textrm{and } \qquad 
     \| w - I_h w \|_{H^1_0(\Omega)} \preceq  h^{\beta-1} \| w \|_{H^\beta(\Omega)}.
  \]

  \item \emph{Stability}: If $w \in \widetilde{H}^\beta(\Omega)$ with $\beta \in (0,\tfrac32)$, then we have
  \[
    \| I_h w \|_{\widetilde{H}^\beta(\Omega)} \preceq \| w \|_{\widetilde{H}^\beta(\Omega)}.
  \]
\end{enumerate}
where, in all estimates, the hidden constants depend only on the shape-regularity of the mesh and the constants in the last two inequalities also depend on the quasi-uniformity.
\end{proposition}
\begin{proof}
The positivity follows from its definition, see \cite{ChenNochetto}.

The $L^2(\Omega)$--approximation property of $I_h$ is derived as follows.
From \cite[Lemma 3.2]{ChenNochetto}, we have that
\[
  \| w - I_h w \|_\Ldeux \preceq h \|\GRAD w\|_\Ldeux, \quad \forall w \in \Hunz,
\]
and that
\[
  \| w - I_h w \|_\Ldeux \preceq h^2 \|D^2w\|_\Ldeux, \quad \forall w \in \Hunz \cap \Hdeux.
\]
Consequently, interpolating these results we obtain that for $\beta \in [1,2]$
\[
  \| w - I_h w \|_\Ldeux \preceq h^\beta \| w\|_{H^\beta(\Omega)}, \quad \forall w \in \Hunz \cap H^\beta(\Omega).
\]

We now discuss the $\Vss$--approximation properties. Since we have already established the $\Ldeux$--approximation property, it suffices to focus on $\Hunz$. This estimate follows from its stability and the $\Ldeux$--approximation property. 
Indeed, let  $S_h:H^1_0(\Omega) \rightarrow \mathbb V_h$ be the the Scott-Zhang operator \cite{MR1011446} and use an inverse inequality to write
\[
  \| \GRAD(w-I_hw)\|_\Ldeux \preceq  \| \GRAD(w - S_h w) \|_\Ldeux + h^{-1} \left( \| w - S_h \|_\Ldeux + \| w - I_h w \|_\Ldeux \right).
\]
The $\Hunz$--approximation property now follows from the approximation properties of $S_h$ in $L^2(\Omega)$ and $H^1_0(\Omega)$ and those of $I_h$ in $L^2(\Omega)$.

To show the, final, stability property we proceed as follows
\begin{align*}
  \| I_h w \|_{\widetilde{H}^\beta(\Omega)} &\leq \| S_h w \|_{\widetilde{H}^\beta(\Omega)} + \| I_h w -S_h w \|_{\widetilde{H}^\beta(\Omega)} \\
    &\preceq \| w \|_{\widetilde{H}^\beta(\Omega)} + h^{-\beta} \| I_h w - S_h w \|_\Ldeux \\
    &\preceq \| w \|_{\widetilde{H}^\beta(\Omega)} + h^{-\beta}\| w - I_h w \|_\Ldeux + h^{-\beta}\| w - S_h w \|_\Ldeux \\
    &\preceq \| w \|_{\widetilde{H}^\beta(\Omega)} + h^{-\beta}\| w - I_h w \|_\Ldeux,
\end{align*}
where we used an inverse inequality between $\widetilde{H}^\beta(\Omega)$ and $\Ldeux$ and the stability and approximation properties on fractional Sobolev spaces of $S_h$ \cite[Lemma 7.6]{BLP17}. 
It remains to invoke the already proven $L^2(\Omega)$--approximation estimate.
Notice that the inverse inequality used above holds thanks to the norm equivalence property
\[
	\|v_h\|_{\widetilde H^{\beta}(\Omega)}\asymp \|v_h\|_{\dot H^{\beta}(\Omega)} 
	\asymp \|v_h\|_{\dot H_h^\beta(\Omega)}, \quad  v_h\in \mathbb V_h, \quad \beta\in [0,\tfrac32),
\]
 discussed in  Remark~\ref{r:norm-equivalence} and in Proposition 3.10 of \cite{XuThesis}; see also \cite{MR3671494}. 
 Here  
\[
	\|v_h\|_{\dot H_h^\beta(\Omega)} 
	:=\left(\sum_{j=1}^{\mathcal M_h} \lambda_{j,h}^\beta |(v_h,\psi_{j,h})|^2\right)^{1/2},
\]
$\mathcal M_h$ denotes the dimension of $\mathbb V_h$ and $\{\lambda_{j,h},\psi_{j,h}\}$ is the set of discrete eigenpairs of the Dirichlet form, \ie
\[
	(\GRAD\psi_{j,h},\GRAD\phi_h)_{L^2(\Omega)} 
	= \lambda_{j,h}(\psi_{j,h},\phi_h)_{L^2(\Omega)}, \quad\forall \phi_h\in \mathbb V_h .
\]
\end{proof}

The Chen-Nochetto interpolant $I_h$ allows us to define the discrete admissible set
\[
  \polK_h := \left\{ w_h \in \polV_h : w_h \geq I_h \chi, \ \mae \text{ in } \Omega \right\};
\]
compare with \eqref{e:K}.
Observe that 
\begin{equation}\label{e:in_set}
w \in \calK \qquad \textrm{implies }\qquad I_h w \in \polK_h.
\end{equation}

\subsection{Numerical approximation of $a_s$}
\label{sub:DunfordTaylorreview}

The nonlocal operator $\Laps$ included in $\mathcal A$ involves the integration of a singular kernel over all of $\Rd$.
For its approximation, we proceed with a discrete bilinear form as originally proposed in \cite{BLP17}.
The main idea behind this approach is the equivalent representation of the bilinear form $a_s$ that was shown in \cite[Theorem 4.1]{BLP17}
\begin{equation}
\label{eq:chvarLaps}
  a_s(v, w) = \frac{2 \sin(\pi s)}\pi \int_0^\infty t^{2-2s} (-\LAP(I-t^2\LAP)^{-1}  \widetilde v,  \widetilde w)_{L^2(\Rd)} \frac{\diff t}t, \quad v,w \in \widetilde H^{s}(\Omega),
\end{equation}
where the operators $\LAP$ and $(I-t^2\LAP)^{-1}$ inside the integrals are acting on functions defined over $\Rd$ so that the inverse is understood in Fourier sense, \ie
\[
  \calF\left( (I-t^2\LAP)^{-1} w \right) = \frac1{1+t^2|\xi|^2} \calF(w).
\]
For $w \in L^2(\Rd)$ let us now denote $\eta_w(t) := -t^2\LAP(I-t^2\LAP)^{-1}w$. The numerical scheme developed in \cite{BLP17} proceeds in three steps:

\begin{enumerate}[1.]
  \item \emph{Sinc quadrature}: We introduce the change of variables $t = e^{-y/2}$ in \eqref{eq:chvarLaps} and apply a truncated equally spaced quadrature. Let $k>0$ and set
  \[
    y_j := jk, \ j \in [-N^-,N^+] \cap \polZ, \quad N^+ := \left\lceil \frac{\pi^2}{2k^2(1-s)} \right\rceil, \quad N^- := \left\lceil \frac{\pi^2}{4sk^2} \right\rceil,
  \]
  to obtain the approximate bilinear form on $\widetilde H^s(\Omega)$
  \begin{equation}
  \label{eq:defofak}
    a_s^k(v, w) := \frac{\sin(\pi s)k}\pi \sum_{j=-N^-}^{N^+} e^{sy_j} ( \eta_v(e^{-y_j/2}), \widetilde w)_{L^2(\Rd)}.
  \end{equation}
  We refer to \cite{MR1171217} for a review of the sinc quadrature and to \cite{bonito2017sinc} for their approximations for these specific integrals. 

  \item \emph{Truncation}: The representation \eqref{eq:defofak} involves the computation of $\eta_v$ via a partial differential equation defined over $\Rd$. We approximate this function by the solution of an associated problem defined on a bounded domain. Let $B$ the unit ball of $\Rd$. Recall that, by assumption $\overline \Omega \subset B$. For a parameter $M$ we define the dilated domains
  \begin{equation}
  \label{eq:defofBM}
    B^M(t) = \begin{dcases}
              \left\{ (1+t(1+M))x : x \in B \right\}, & t \geq 1, \\
              \left\{ (2+M)x : x \in B \right\}, & t< 1.
             \end{dcases}
  \end{equation}
  Upon noticing that, for any $ w\in \Ldeux $, we can equivalently write $\eta_w(t) = \widetilde w -(I-t^2\LAP)^{-1} \widetilde w$, we approximate $\eta_w$ by $\eta_w^M := \widetilde w + \xi_w^M(t)$, where $\xi_w^M(t) \in H^1_0(B^M(t))$ solves 
  \begin{equation}
  \label{eq:defofxiM}
    \int_{B^M(t)} \left( \xi_w^M(t) \phi + t^2 \GRAD \xi_w^M(t)\GRAD \phi \right) \diff x = - \int_\Omega w \phi \diff x, \ \forall \phi \in H^1_0(B^M(t)).
  \end{equation}
  These considerations give rise to the following bilinear form on $\widetilde H^s(\Omega)$:
  \begin{equation}
  \label{eq:defofoakM}
    a^{k,M}_s(v, w) := \frac{\sin(\pi s)k}\pi \sum_{j=-N^-}^{N^+} e^{sy_j} (\eta_v^M(e^{-y_j/2}),\widetilde w)_{L^2(B^M(e^{-y_j/2}))}.
  \end{equation}

  \item \emph{Discretization}: It remains to discretize problem \eqref{eq:defofxiM} in space. For a fixed $t$, we let $\calT_h(t)$ be a conforming shape-regular and quasi-uniform triangulation of $B^M(t)$ made of simplices (possibly curved to match the boundary of $B^M(t)$). We require that $\calT_h(t)$ restricted to $\overline \Omega$ coincides with $\calT_h(\Omega)$. Over $\calT_h(t)$ we define $\polV_h^M(t)$ to be the space of piecewise affine functions subordinate to $\calT_h(t)$, that vanish on $\partial B^M(t)$. Notice that, if $w_h \in \polV_h$, then $\widetilde w_h \in \polV_h^M(t)$. We thus approximate \eqref{eq:defofxiM} by $\xi_{h,w}^M(t) \in \polV_h^M(t)$ that solves
  \begin{equation}
  \label{eq:defofxiMh}
    \int_{B^M(t)} \left( \xi_{h,w}^M(t) \phi_h + t^2 \GRAD \xi_{h,w}^M(t)\GRAD \phi_h \right) \diff x = - \int_\Omega w \phi_h \diff x, \ \forall \phi_h \in \polV_h^M(t).
  \end{equation}
  This gives rise to the fully discrete bilinear form on $\mathbb V_h$
  \begin{equation}
  \label{eq:defoffinala}
    a_{s,h}^{k,M}(v_h,w_h) := \frac{\sin(\pi s)k}\pi \sum_{j=-N^-}^{N^+} e^{sy_j} ( \eta_{h,v_h}^M(e^{-y_j/2}), \widetilde w_h )_{L^2(B^M(t))}
  \end{equation}
  with $\eta_{h,v_h}^M:=\widetilde {v}_h+\xi_{h,v_h}$.
\end{enumerate}

We end this section by recalling properties of the bilinear form $a_{s,h}^{k,M}$ used in the analysis below.
The consistency error incurred in approximating the bilinear form $a_s$ by its fully discrete (and computable) counterpart $a_{s,h}^{k,M}$ is analyzed in \cite{BLP17}: for $\beta\in(s,\tfrac32)$ we have that
\begin{equation}
\label{eq:consistencyerror}
  \sup_{0 \neq v_h, 0 \neq w_h \in \polV_h} \frac{\left| a_s(v_h, w_h) - a_{s,h}^{k,M}(v_h,w_h) \right|}{ \| v_h \|_{\widetilde H^\beta(\Omega)} \| w_h \|_{\tHs}}
    \preceq e^{-c_1/k} + e^{-c_2M} + h^{\beta-s} |\log h| .
\end{equation}
It is also possible to show, see \cite[Theorem 7.2]{BLP17}, that provided the sinc-quadrature spacing $k$ is sufficiently small, the bilinear form $a_{s,h}^{k,M}$ is coercive on $\polV_h \subset \widetilde H^\beta(\Omega)$ for all $\beta \in [0,\tfrac32)$.
More precisely, if $C$ denotes the implicit constant in \eqref{eq:consistencyerror} and we assume that 
\begin{equation}\label{e:quad_mesh}
Ce^{c_1/k}h^{s-1}<1, 
\end{equation}
then we have
\begin{equation}
\label{eq:disccoercive}
  \| w_h \|_{\tHs}^2 \preceq a_{s,h}^{k,M}(w_h,w_h), \quad \forall w_h \in \polV_h,
\end{equation}
where the implicit constant does not depend on $h$.

\subsection{The numerical scheme and its error analysis}
\label{sub:Thescheme}

We are now in position to define a computable discrete bilinear form approximating $\mathcal A$. For $v_h, w_h \in \polV_h \times \polV_h$ we set
\[
  \calA_h(v_h,w_h) := \bit \calL(v_h,w_h) +  \calD(v_h,w_h) + a_{s,h}^{k,M}(v_h,w_h),
\]
where $a_{s,h}^{k,M}$ is the bilinear form defined in \eqref{eq:defoffinala}. 
This bilinear form is continuous. It is also coercive, namely 
\begin{equation}
\label{eq:Ahcoercive}
  \| w_h \|_{\Vss}^2 \preceq \calA_h(w_h,w_h), \quad \forall w_h \in \polV_h,
\end{equation}
with an implicit constant that is independent of $h$, provided the quadrature spacing $k$ satisfies \eqref{e:quad_mesh} for the coercivity \eqref{eq:disccoercive} of $a_{s,h}^{k,M}$ to hold.  

With this notation the discrete obstacle problem reads: find $u_h \in \polK_h$ such that
\begin{equation}
\label{eq:VIh}
  \calA_h(u_h, u_h-v_h) \leq (f, u_h-v_h)_\Ldeux, \quad \forall v_h \in \polK_h.
\end{equation}
Once again, the Lions--Stampacchia theorem ensures the existence and uniqueness of a solution $u_h \in \polK_h$.

The regularity results developed in Section~\ref{sec:regularity} are now brought into play to derive estimates on the error $\| u - u_h \|_{\Vss}$.
Recall that Theorem~\ref{thm:regularityVI} guarantees $u\in H^\sigma(\Omega)$, where $\sigma$ is given by \eqref{eq:defofsigma}.
Therefore, we expect from interpolation theory (Part 3 of Proposition~\ref{prop:properitesIh}) a rate of convergence when measuring the error in the $\Vss$--norm to be
\begin{equation}\label{e:sigmaStar}
\sigma^*:= 
  \sigma^*(\bit,\bbeta,s, r) = \begin{dcases}
                              \min\left\{s,\left(\frac12\right)^-\right\}, & \text{Cases~\ref{case:frac} and \ref{case:fracdrift}} , \\                           
                              \mu(s,r)-1, & \text{Case~\ref{case:diff}},
                            \end{dcases}                   
\end{equation}
where $\mu$ is defined in \eqref{eq:defofmu}.
However, the convergence of the proposed algorithm is restricted by the consistency error discussed above. 
This is the object of the next result. 

\begin{theorem}[rate of convergence]
\label{thm:convergence_rate}
Suppose that Assumptions~\ref{a:obstacle}, \ref{a:mainassumption_regularity} and \ref{a:regLinear} hold. Assume in addition that, for Case~\ref{case:fracdrift}, we have that $s \in (\tfrac{d+1}6,1)\cap (\tfrac12,1)$.
Let $u \in \Vss$ be the solution to \eqref{eq:VI} and $u_h \in\polV_h$ the solution to the discrete counterpart \eqref{eq:VIh}. 
In addition, assume that $k \asymp  |\log(h)|$ and $M = |\log(h)|$ are such that \eqref{e:quad_mesh} holds. In this setting, and with this notation, we have 
\[
  \| u - u_h \|_\Vss \preceq h^{\min\{ \sigma^*,(\frac32-s)^-\}} |\log h| \left( \| f \|_{L^2(\Omega)} + \| \max\{F,0\}\|_{L^2(\Omega)}+\| \chi \|_{H^\sigma(\Omega)}\right) .
\]
\end{theorem}
\begin{proof}
We proceed in several steps.

\noindent \boxed{1} Let $v_h \in \polK_h$. The discrete coercivity \eqref{eq:Ahcoercive}, the continuity of $\mathcal A(.,.)$ and the discrete obstacle system  \eqref{eq:VIh}  satisfied by $u_h$ yield
\begin{align*}
  \| v_h-u_h \|_\Vss^2 &\preceq \calA_h( v_h - u_h, v_h - u_h) 
    \preceq (\calA - \calA_h)(v_h, u_h - v_h)  + \calA(u - v_h,u_h - v_h) \\
    &\quad + \calA_h(u_h, u_h - v_h) - \calA(u, u_h - v_h ) \\
    &\preceq (\calA - \calA_h)(v_h, u_h - v_h) + \|u - v_h\|_\Vss \|u_h - v_h\|_\Vss \\
    &\quad + (f, u_h - v_h)_\Ldeux - \calA(u, u_h - v_h ).
\end{align*}
Incorporating the definition \eqref{eq:defofLambda} of the Lagrange multiplier $\Lambda$ as well as the definition of  the forms $\calA$ and $\calA_h$, we arrive at
\begin{equation*}
  \| u - u_h \|_\Vss^2 \preceq \| u - v_h \|_\Vss^2 + (a_s - a_{s,h}^{k,M})(v_h , u_h - v_h)  + \scl \Lambda, v_h - u_h \scr_{\Vss',\Vss},
\end{equation*}
for every $v_h \in \polK_h$.
We fix $v_h = I_h u$ and invoke the interpolation properties of $I_h$ obtained in Proposition~\ref{prop:properitesIh}, in conjunction with the regularity estimates $u \in H^\sigma(\Omega)$ of Theorem~\ref{thm:regularityVI}, to deduce that
\begin{equation}\label{i:error-split}
  \| u - u_h \|_\Vss^2 \preceq h^{2\sigma^*}\| u \|^2_{H^{\sigma}(\Omega)} + (a_s - a_{s,h}^{k,M})(I_h u, u_h - I_h u)  + \scl \Lambda, I_h u - u_h \scr_{\Vss',\Vss},
\end{equation}
where $\sigma^*$ is given by \eqref{e:sigmaStar}.

\noindent \boxed{2} We now estimate the second term on the right and side of \eqref{i:error-split}.
It directly relates to the consistency error \eqref{eq:consistencyerror} and satisfies for  $k \asymp  |\log(h)|$, $M \asymp |\log(h)|$  and $\beta=\min\{\sigma,\left(\frac32\right)^-\}$
$$
  (a_s - a_{s,h}^{k,M})(I_hu, u_h - I_hu) \preceq \\
  h^{\min\{\sigma,\left(\frac32\right)^-\}-s}|\log h| \| I_h u \|_{\widetilde H^{\min\{\sigma,\left(\frac32\right)^-\}}(\Omega)} \| u_h -I_h u \|_{\tHs}.
$$
Since Proposition~\ref{prop:properitesIh} gives us stability and interpolation error estimates for $I_h$, and Remark~\ref{r:norm-equivalence} gives a norm equivalence property, we are able to obtain that
\begin{align*}
  (a_s - a_{s,h}^{k,M})(I_hu, u_h - I_hu) \preceq & (1+\frac 1 \epsilon) h^{2(\min\{\sigma^*,(\frac32-s)^-\}) }|\log h|^2\| u \|_{H^{\sigma}(\Omega)}^2 \\
  & + \epsilon \| u - u_h \|_{\widetilde H^s(\Omega)}^2,
\end{align*}
for every $\epsilon>0$. Notice that we used the relation $\sigma^* \leq \sigma-s$.
Returning to \eqref{i:error-split} we obtain
\begin{equation}\label{i:error-split_2}
  \| u - u_h \|_\Vss^2  \preceq  h^{2(\min\{\sigma^*,(\frac32-s)^-\}) }|\log h|^2 \| u \|_{H^\sigma(\Omega)}^2   + \scl \Lambda, I_h u - u_h \scr_{\Vss',\Vss}.
\end{equation}

\noindent \boxed{3} It remains to bound last term on the right hand side of \eqref{i:error-split_2} involving the Lagrange multiplier $\Lambda$. We notice, first of all, that owing to Theorem~\ref{thm:Lambdaregular}, we can replace the duality pairing here with an $\Ldeux$--inner product. Thus, we write
\begin{align*}
  \scl \Lambda, I_h u - u_h \scr_{\Vss',\Vss} &= (\Lambda, I_h u - u_h)_\Ldeux \\
  &= \left( \Lambda, I_h(u-\chi) - (u-\chi) \right)_{\Ldeux} + \left( \Lambda, u-\chi \right)_{\Ldeux} \\ 
  &\quad + \left( \Lambda, I_h \chi - u_h \right)_{\Ldeux}.
\end{align*}
In addition, from Theorem~\ref{t:continuity-C} we conclude that Theorem~\ref{thm:complconds} holds, and so we have that the compatibility conditions are satisfied. This implies that
\[
  \left(\Lambda, u - \chi \right)_\Ldeux = 0
\]
and that $\Lambda \geq 0$ a.e.~in $\Omega$. In addition, since $u_h \in \polK_h$ implies $I_h\chi - u_h \leq 0$, this leads to
\[
  \left( \Lambda, I_h \chi - u_h \right)_{\Ldeux} \leq 0.
\]
Gathering the above three relations we deduce that
\[
  \left( \Lambda, I_h u - u_h \right)_{\Ldeux} \leq \left( \Lambda, I_h(u-\chi) - (u-\chi) \right)_{\Ldeux}.
\]
To conclude, we once again invoke the interpolation estimates to write
\begin{align*}
  \left( \Lambda, I_h u - u_h \right)_{\Ldeux} & \leq  \| \Lambda \|_{L^2(\Omega)} \| I_h(u-\chi) - (u-\chi)\|_{L^2(\Omega)} \\
 & \preceq h^{\sigma} \left( \| u \|_{H^\sigma(\Omega)} + \| \chi \|_{H^\sigma(\Omega)} \right) \| \Lambda \|_{L^2(\Omega)}.
\end{align*}

\noindent \boxed{4} Since $\sigma^* \leq \sigma/2$, substituting the previous inequality in \eqref{i:error-split_2} yields 
\[
 \| u - u_h \|_\Vss  \preceq  h^{\min\{\sigma^*,(\frac32-s)^-\}}|\log h| \left(\| u \|_{H^\sigma(\Omega)}+\| \chi \|_{H^\sigma(\Omega)} + \| \max\{F,0\} \|_{L^2(\Omega)}\right).
\]
It remains to use  the regularity estimate of Theorem~\ref{thm:regularityVI} and Theorem~\ref{thm:Lambdaregular} to express the right hand side of this estimate in terms of the data. This concludes the proof.
\end{proof}


\section{Numerical illustrations}
\label{sec:numex}

In this section we carry out a series of numerical examples that illustrate and go beyond our theory. 

\subsection{Numerical Implementation}
We implement the numerical algorithm using the \texttt{deal.II} finite element library \cite{dealII90}. For our one dimensional examples we use continuous piecewise linear finite elements subordinate to a uniform subdivision in $\Omega$. In two dimensions, we use bilinear quadrilateral elements subordinate to a regular (in the sense of \cite{MR1930132}) subdivision in $\Omega$. 

\subsubsection{Mesh generation}
We recall that we assume (without loss of generality) that the domain $\Omega$ is a subset of the unit ball $B$. 
We start with a quasi-uniform subdivision $\mathcal T_h$ of $B$ matching $\partial \Omega$ and where $h$ denotes the largest diameter among all the elements in $\mathcal T_h$.
Motivated by the exponential decay of the solution to the elliptic problem \eqref{eq:defofxiM} in the larger ball $B^M(t)$  \cite[Lemma 2.1]{MR1933726}, 
an exponentially graded extension  to $B^M(t)$ of  the subdivision $\mathcal T_h$ is advocated as  in  \cite[Section 8.2]{BLP17}.
Notice that such subdivisions violate the shape-regularity and quasiuniformity conditions required in step 3 of Section~\ref{sub:DunfordTaylorreview}.
However, the advantage of such non-uniform partitions is to keep the dimension of $\mathbb V^M(t)$ approximatively constant in $t$. 

\subsubsection{The discrete problem}
Let $\mathcal M_{h,t}$ be the dimension of $\mathbb V^M_h(t)$ and recall that $\mathcal M_h$ is the dimension of $\mathbb V_h$. Let $\utilde\Psi$ and $\utilde F\in \Real^{\mathcal M_h}$ be the coefficient vectors of $I_h\chi$ and the $L^2(\Omega)$ projection of $f$ onto $\mathbb V_h$, respectively. We want to find the discrete solution $\utilde {U}\in \Real^{\mathcal M_h}$ and the discrete Lagrange multiplier $\utilde \Lambda \in \Real^{\mathcal M_h}$ satisfying
\[
\begin{aligned}
	&\utilde S \utilde U + \utilde \Lambda =\utilde F,\\
	&\utilde U_i \ge \utilde\Psi_i, \quad \utilde\Lambda_i\ge 0,\quad\text{and}\quad \utilde\Lambda_i(\utilde U_i-\utilde\Psi_i) = 0,
	\quad \text{for } i=1,2,\ldots,\mathcal M_h .
\end{aligned}
\]
Here $\utilde S$ is the system matrix corresponding to the bilinear form $\calA_h$ and is given by
\[
	\utilde S = \sigma\utilde A_0 + \utilde A_{\bbeta}
	+\frac{\sin{(\pi s)k}}{\pi}\utilde M_0 \utilde R \sum_{i=-N^-}^{N^+}e^{s y_i}(e^{y_i}\utilde M_i+\utilde A_i)^{-1} \utilde A_i \utilde E ,
\]
where 
\begin{enumerate}[$\bullet$]
	\item $\utilde A_0, \utilde M_0, \utilde A_{\bbeta}\in \Real^{\mathcal M_h\times \mathcal M_h}$ are the stiffness, mass and advection matrices in the finite element space $\mathbb V_h$;
	\item $\utilde A_i, \utilde M_i \in \Real^{\mathcal M_{h,t}\times \mathcal M_{h,t}}$ are stiffness and mass matrices in the finite element space $\mathbb V_h^M(t)$;
	\item $\utilde E:\Real^{\mathcal M_h} \rightarrow \Real^{\mathcal M_{h,t}}$ is the zero extension operator and $\utilde R: \Real^{\mathcal M_{h,t}} \rightarrow \Real^{\mathcal M_h}$ is the  restriction operator.
\end{enumerate}

The above discrete problem is solved with the primal-dual active set method \cite[Section 5.3]{MR3309171} briefly recalled now. Let $(\utilde U^0, \utilde \Lambda^0)\in \Real^{\mathcal M_h}\times \Real^{\mathcal M_h}$ and $\rho$ be a positive constant.
Compute iteratively $(\utilde U^{k+1}, \utilde \Lambda^{k+1})$, $k\geq 0$, as the solution to 
\begin{equation}\label{e:discrete-system}
	\left( 
	\begin{aligned}
	&\utilde S & (\utilde I^k)^{\Tr} \\
	&\utilde I^k & 0
	\end{aligned}
	\right)
	\left(
	\begin{aligned}
	& \utilde U^{k+1}\\
	& \utilde \Lambda^{k+1}
	\end{aligned}
	\right)
	=
	\left(
	\begin{aligned}
	& \utilde F\\
	& \utilde I^k \utilde\Psi
	\end{aligned}
	\right),
\end{equation}
where  $\utilde I^k \in \mathbb R^{|\mathscr A^k|\times  \mathcal M_h}$ is defined by
\[
	(\utilde I^k)_{ij}=
	\begin{dcases}
	1,&\text{if }j=\mathscr A^{k}_i ,\\
	0,&\text{otherwise}.
	\end{dcases}
\]
and $\mathscr A^{k}$ is the vector of ordered current  active set of indices given by
\[
	\mathscr A^{k}_i := \argmin_{\substack{ \utilde\Lambda_j+\rho(\utilde U^{k}-\utilde \Psi)_j < 0  \\ 
	\mathscr A^k_l \not = j, ~ l<j }} j.
\]

Given a tolerance $\epsilon_{\text{stop}}$, we stop the iteration process when  $\|\utilde U^{k+1}-\utilde U^k\|_{h,\bit}<\epsilon_{\text{stop}}$, where for $w_h\in\mathbb V_h$,
\[
	\|w_h\|_{h,\bit}:=\left(a^{k,M}_{s,h}(w_h,w_h)+\bit\|\nabla w_h\|_\Ldeux^2\right)^{1/2} .
\]

The discrete system \eqref{e:discrete-system} is solved using a Schur complement method, \ie we determine $\utilde \Lambda^{k+1}$ via
\begin{equation}\label{e:lambda-system}
	[\utilde I^k \utilde S^{-1} (\utilde I^k)^{\Tr} ]\Lambda^{k+1} = \utilde I^k (\utilde S^{-1} F - \utilde\Psi) 
\end{equation}
and then we compute $\utilde U^{k+1}$ from
\begin{equation}\label{e:u_system}
	\utilde U^{k+1} = \utilde S^{-1} [\utilde F - (\utilde I^k)^{\Tr} \utilde\Lambda^{k+1}] .
\end{equation}

The evaluation of $\utilde S^{-1}$ in \eqref{e:lambda-system} and \eqref{e:u_system} is approximated using a preconditioned conjugate gradient (when $\bbeta \equiv \boldsymbol0$) or  BI-CGSTAB (when $\bbeta \not = \boldsymbol0)$.
Depending on the value of $\bit$, different preconditioners are applied. 
When $\bit=0$ (Cases~\ref{case:frac} and \ref{case:fracdrift}), the bilinear form $\calA(\cdot,\cdot)$ is equivalent to the $\tHs$ norm squared and we use the inverse of the discrete spectral fractional Laplacian; see \cite{MR3356020} and \cite[Section 8.2]{BLP17} for details. 
Otherwise, when $\bit = 1$ or Case~\ref{case:diff}, we use the multilevel preconditioner introduced in \cite{MR1651742}: Let $j$ be the mesh level and $\phi_i$ for $i=1,\ldots, \mathcal M_{h_j}$ be the nodal basis for $\mathbb V_{h_j}$. We define a sequence of approximation operators $\widetilde Q_j: L^2(\Omega)\to \mathbb V_{h_j}$ by
\[
	\widetilde Q_j w := \sum_{i=1}^{\mathcal M_{h_j}} \frac{(w,\phi_i)_{L^2(\Omega)}}{(1,\phi_i)_{L^2(\Omega)}} \phi_i .
\]
If $J$ denotes the finest mesh level, the preconditioner is given by 
\[
	B_J := \sum_{j=1}^{J-1} (\bar{A} h_j^{-2}+ h_j^{-2s})^{-1} (\widetilde Q_{j+1}-\widetilde Q_j)^2 ,
\]
where $\bar A$ is a constant related to the magnitude of the diffusion coefficient matrix $A$.

System \eqref{e:lambda-system} is solved, again, with an iterative scheme. We use conjugate gradients ($\bbeta \equiv \boldsymbol 0$) or  BI-CGSTAB ($\bbeta \not = \boldsymbol0)$, but this time without preconditioner. 

\subsection{One dimensional convergence tests}
\label{sub:convergence}
Set $\Omega = (-1,1)$, $\chi(x) = 3-6x^2$ and $f(x) = 1$ and the bilinear form $\calL(\cdot,\cdot)$ to be the Dirichlet form \eqref{e:dirichlet-form}. 
The initial subdivision consists of two elements of equal sizes so that $h_0 = \frac 1 2$ and $h_j=h_0/2^j$, $j=1,2,...$.
In addition, for Cases~\ref{case:fracdrift} and \ref{case:diff} we will set $\bbeta = \tfrac12$.


The computation of $a_{s,h}^{k,M}(\cdot,\cdot)$ is carried out with a spacing $k = 0.2$ and truncation parameter $M = 5$ so that the finite element approximation dominates the total error.

\begin{figure}[ht!]
\label{fig:convergence}
  \begin{center}
    \begin{tabular}{ccc}
      \!\!\!\!\includegraphics[scale=.236]{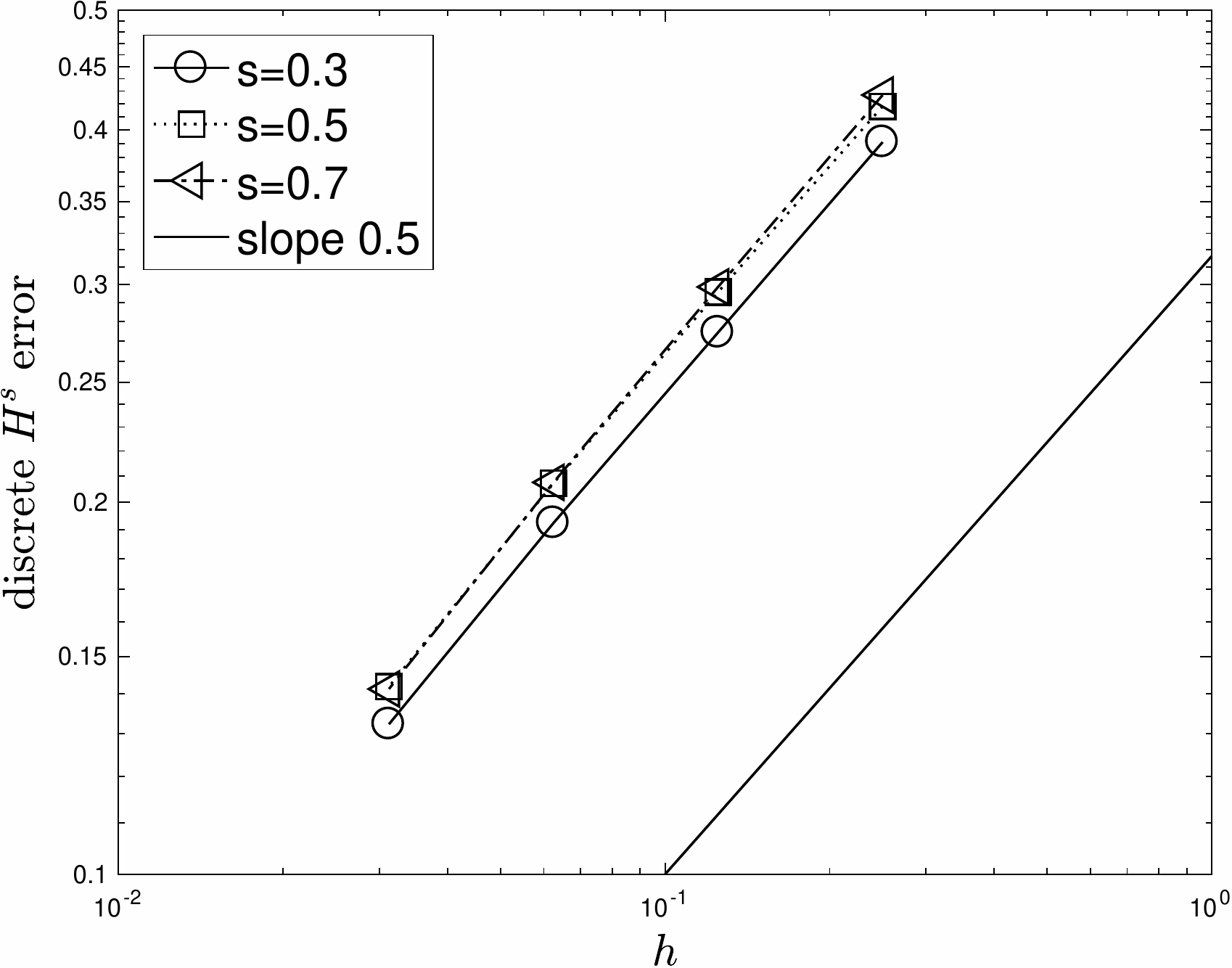}&   \!\!\includegraphics[scale=.236]{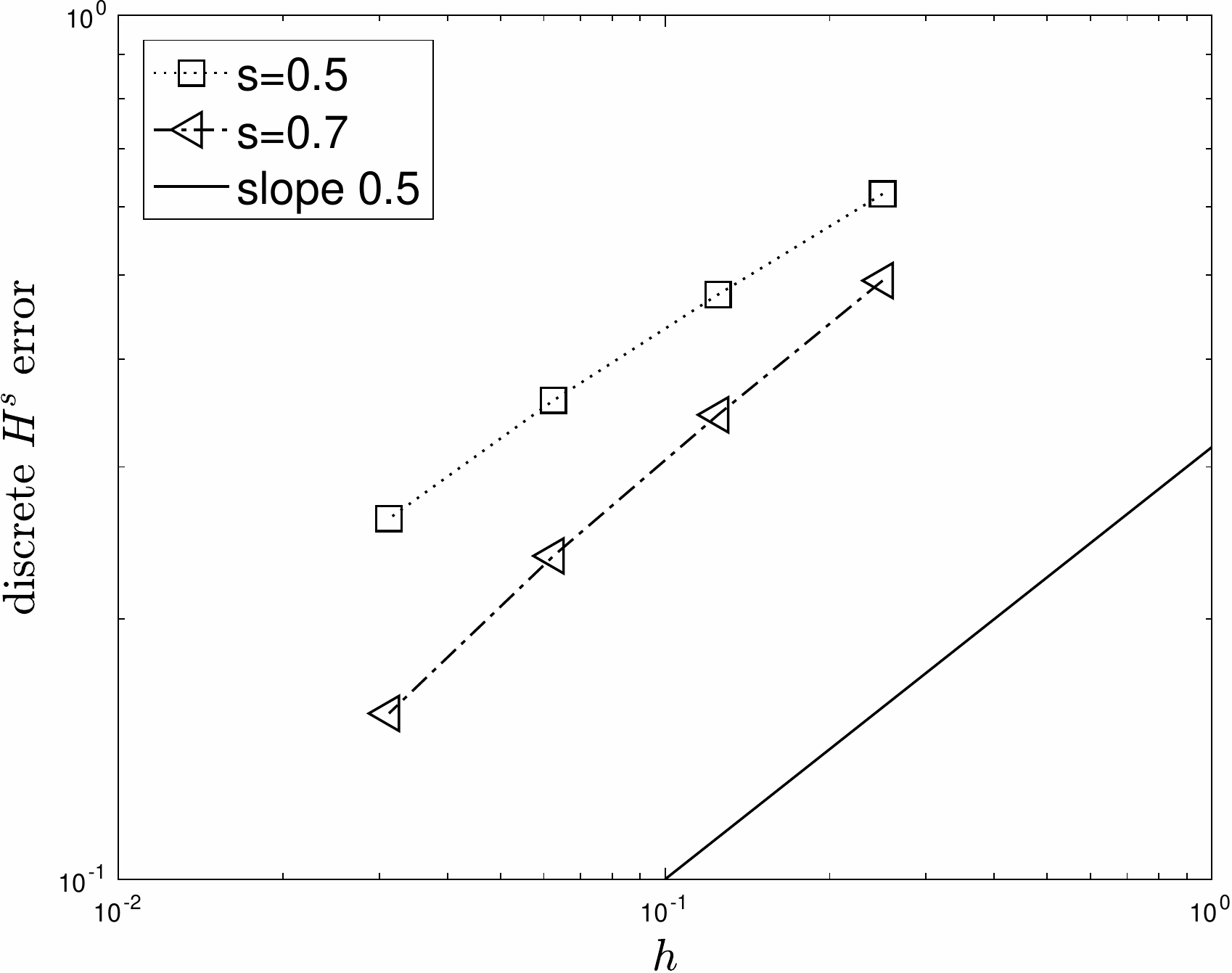} & \!\!\includegraphics[scale=.236]{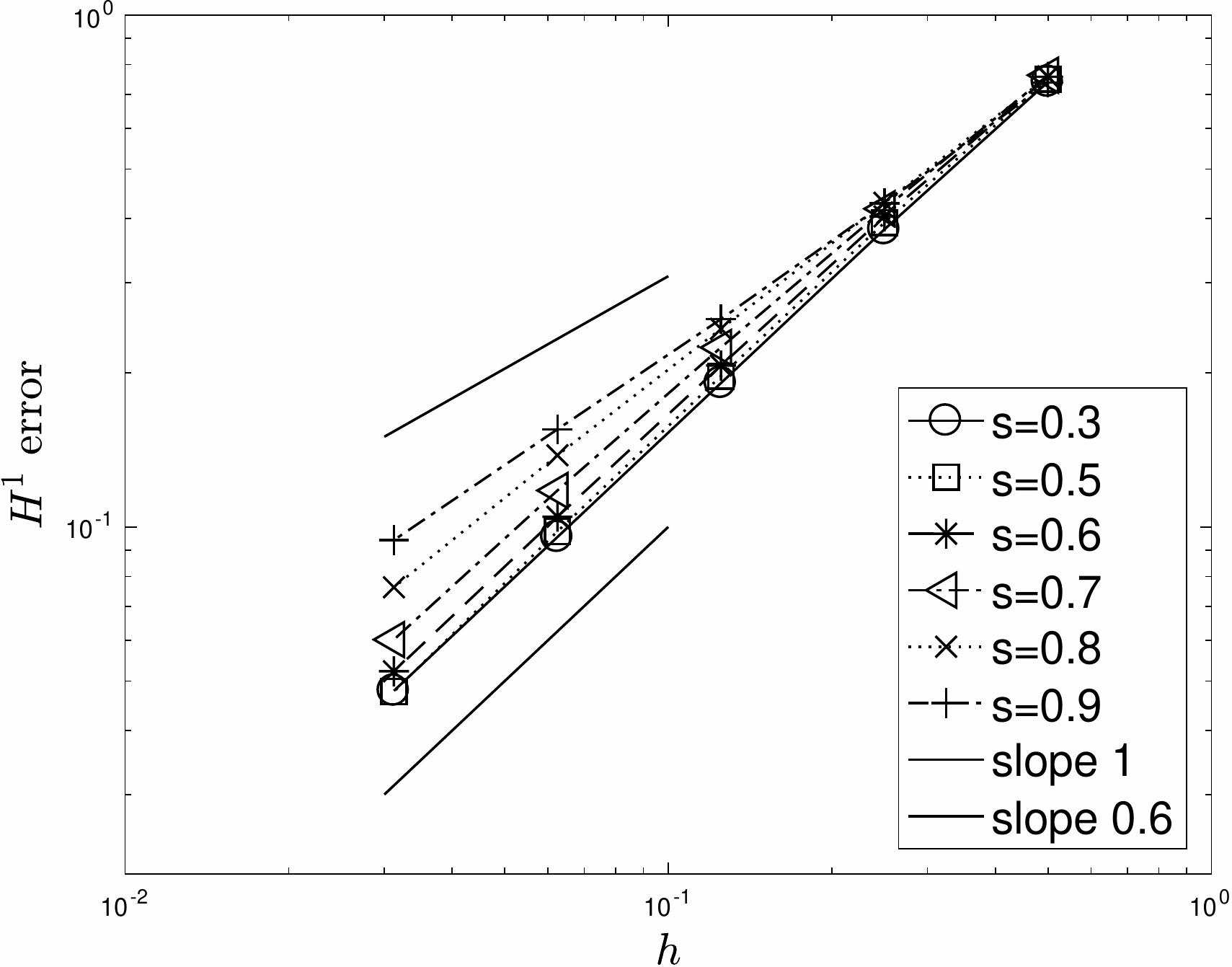}
    \end{tabular}
  \end{center}
\caption{Decay rate of the discrete energy error for the finite element approximation to problem \eqref{eq:VI}. Case~\ref{case:frac} (left), Case~\ref{case:fracdrift} (middle), Case~\ref{case:diff} (right). Note that the case $s=0.5$ for Case ~\ref{case:fracdrift} is not included in the theory developed here.}
\end{figure}

Since the exact solution it is not known to us, as a measure of the error we compute, for $j = 1, \ldots, 4$, the discrete energy error
\begin{equation}\label{e:measure-error}
  e_j:= \| u_{h_j} - u_{\text{ref}} \|_{h,\bit},
\end{equation}
where $u_{\text{ref}}$ is finite element approximation over a very refined mesh. In this case, we set $u_{\text{ref}}=u_{h_9}$. Figure~\ref{fig:convergence} illustrates the decay rate in all the situations and for different values of $s$. 
In the pure fractional diffusion case (left), the observed rates $\mathcal O(h^{1/2})$ matches the prediction of Theorem~\ref{thm:convergence_rate} when $s\ge\tfrac12$. However, this rate is observed as well for $s=0.3$ although Theorem~\ref{thm:convergence_rate} only guarantees $\mathcal O(h^{0.3})$. 
In the case of fractional diffusion with drift (middle), the observed rate of convergence is approximately  $\calO(h^{1/2})$ for $s=0.5,0.7$ as predicted by Theorem~\ref{thm:convergence_rate}. The observed rates for the integro--differential case (right) are in accordance with Theorem~\ref{thm:convergence_rate}. 

\begin{figure}[ht!]
\label{fig:qualitative}
  \begin{center}
    \begin{tabular}{ccc}
      \includegraphics[scale=.09]{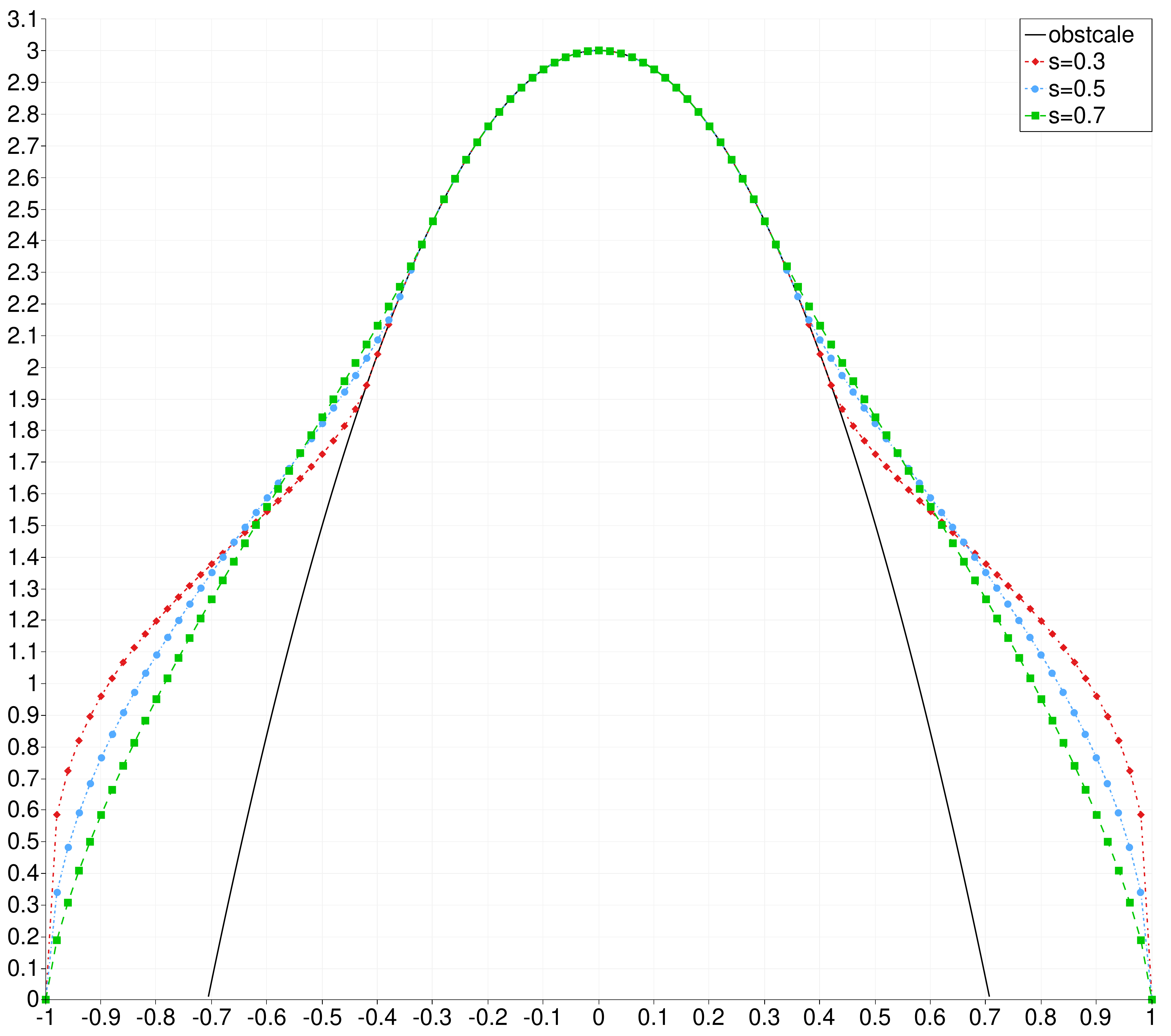}&   \includegraphics[scale=.09]{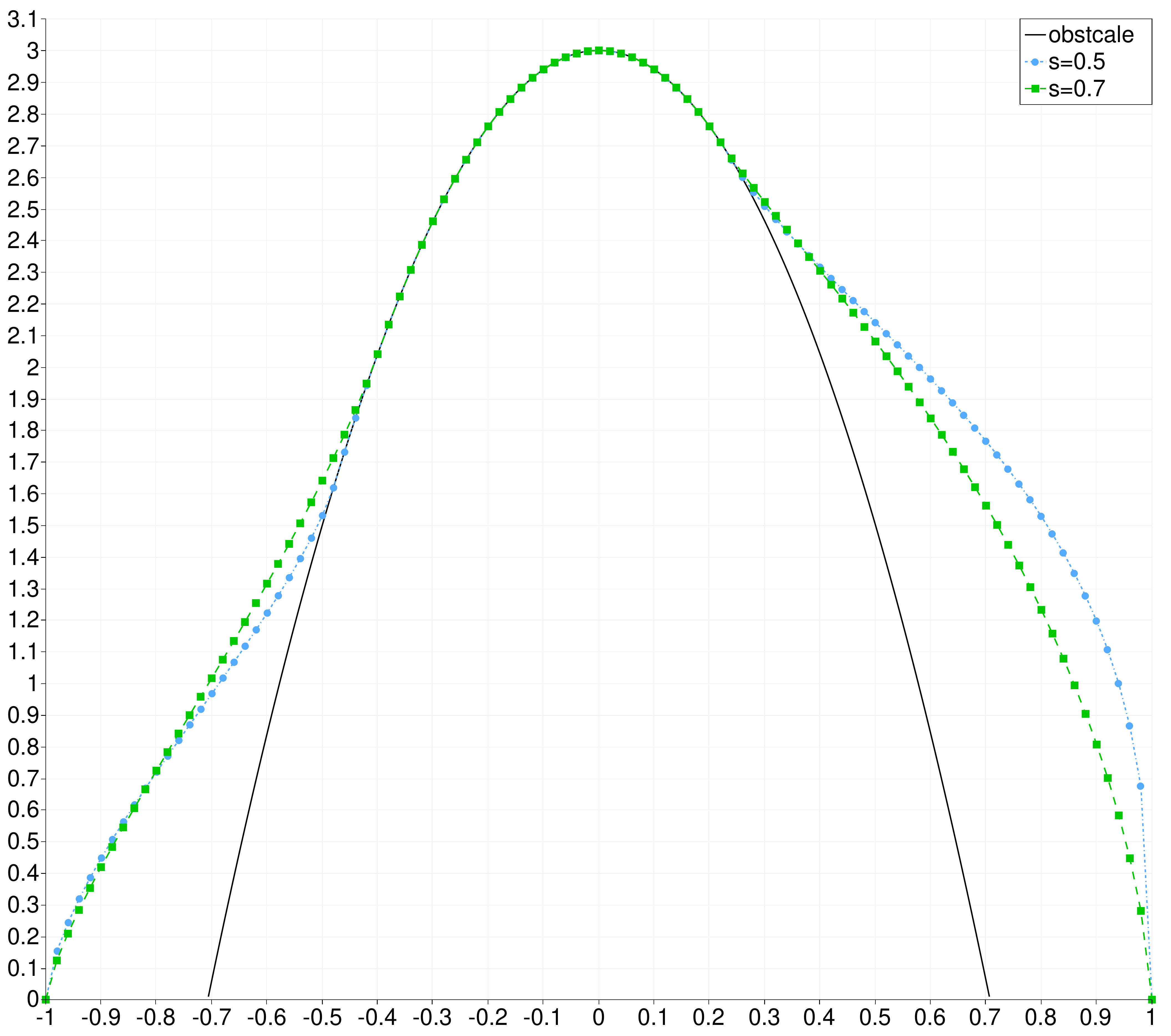} & \includegraphics[scale=.09]{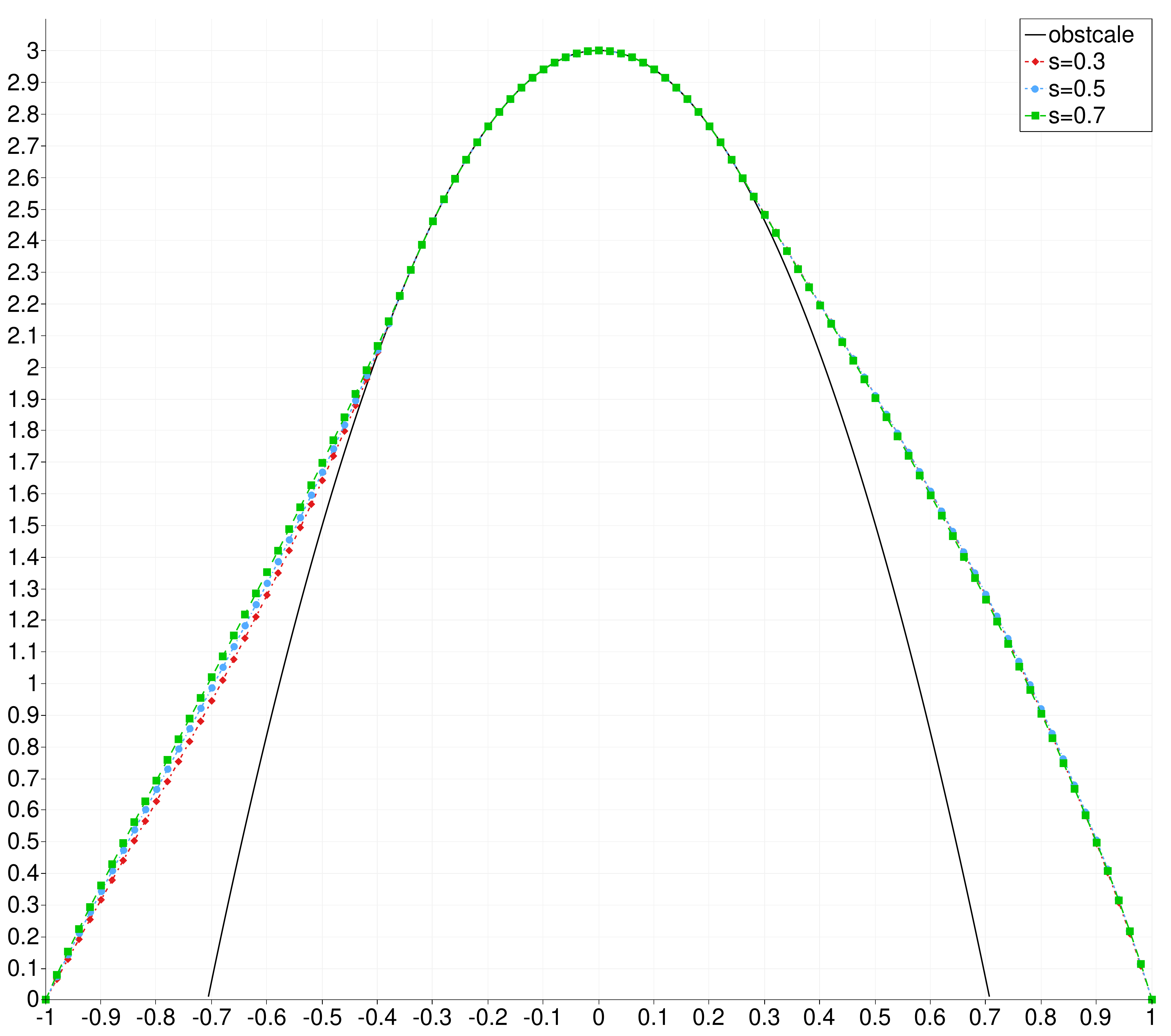}
    \end{tabular}
  \end{center}
\caption{Finite element approximations to \eqref{eq:VI} for Case~\ref{case:frac} (left), Case~\ref{case:fracdrift} (middle), and Case~\ref{case:diff} (right). In each figure, the obstacle is depicted in black (negative part not depicted), the approximate solutions for $s=0.3$ is in red, for $s=0.5$ in blue and for $s=0.7$ in green. Notice that we do not report the case $s=0.3$ when there is a drift, since it falls outside the scope of this work, see Proposition~\ref{prop:drift}.  We also note that the case $s=0.5$ for Case ~\ref{case:fracdrift} is not included in the theory developed here.}
\end{figure}

To appreciate the combined effect of the order of the fractional Laplacian, the drift, and the second order operator, Figure~\ref{fig:qualitative} depicts the solutions in different settings.

\subsection{Two dimensional qualitative experiments}
\label{sub:qualitative2d}

In all the two dimensional examples presented in this section, we compute $a_{s,h}^{k,M}(\cdot, \cdot)$ with $k=0.25$ and $M=4$.

\subsubsection{Unit ball domain}

We set $\Omega$ to be the unit ball, $\chi(x) = 3-6|x|^2$ and, for each case, we consider the following data:

\begin{enumerate}[$\bullet$]
  \item Case \ref{case:frac}, pure fractional diffusion: $f \equiv 1$. The results are shown in Figure~\ref{fig:disk}.
  \item Case \ref{case:fracdrift}, fractional diffusion with drift: $\bbeta = (-\tfrac12,0)\Tr$, and
  \[
    f(x,y) = \begin{dcases}
               2, & (x-\tfrac12)^2 + y^2 <\tfrac14, \\ 0, & (x-\tfrac12)^2 + y^2 \geq \tfrac14.
             \end{dcases}
  \]
  The approximate solution is shown in Figure~\ref{fig:2ddrift}.
  
  \item Case \ref{case:diff}, integro--differential case: $A=0.3 \calI$, $c \equiv 0$, $\bbeta = (-\tfrac12,0)\Tr$, and  $f \equiv 1$. The approximate solution is shown in Figure~\ref{fig:2dintegrodifferential}.
\end{enumerate}

The coarse subdivision of $\Omega$ is described in \cite{BLP17} and uniform refinements are performed to create a sequence of meshes $\mathcal T_{h_j}$, $j \geq 1$. 

\begin{figure}[hbt!]
\label{fig:disk}
  \begin{center}
    \begin{tabular}{cccc}
   \includegraphics[scale = 0.1]{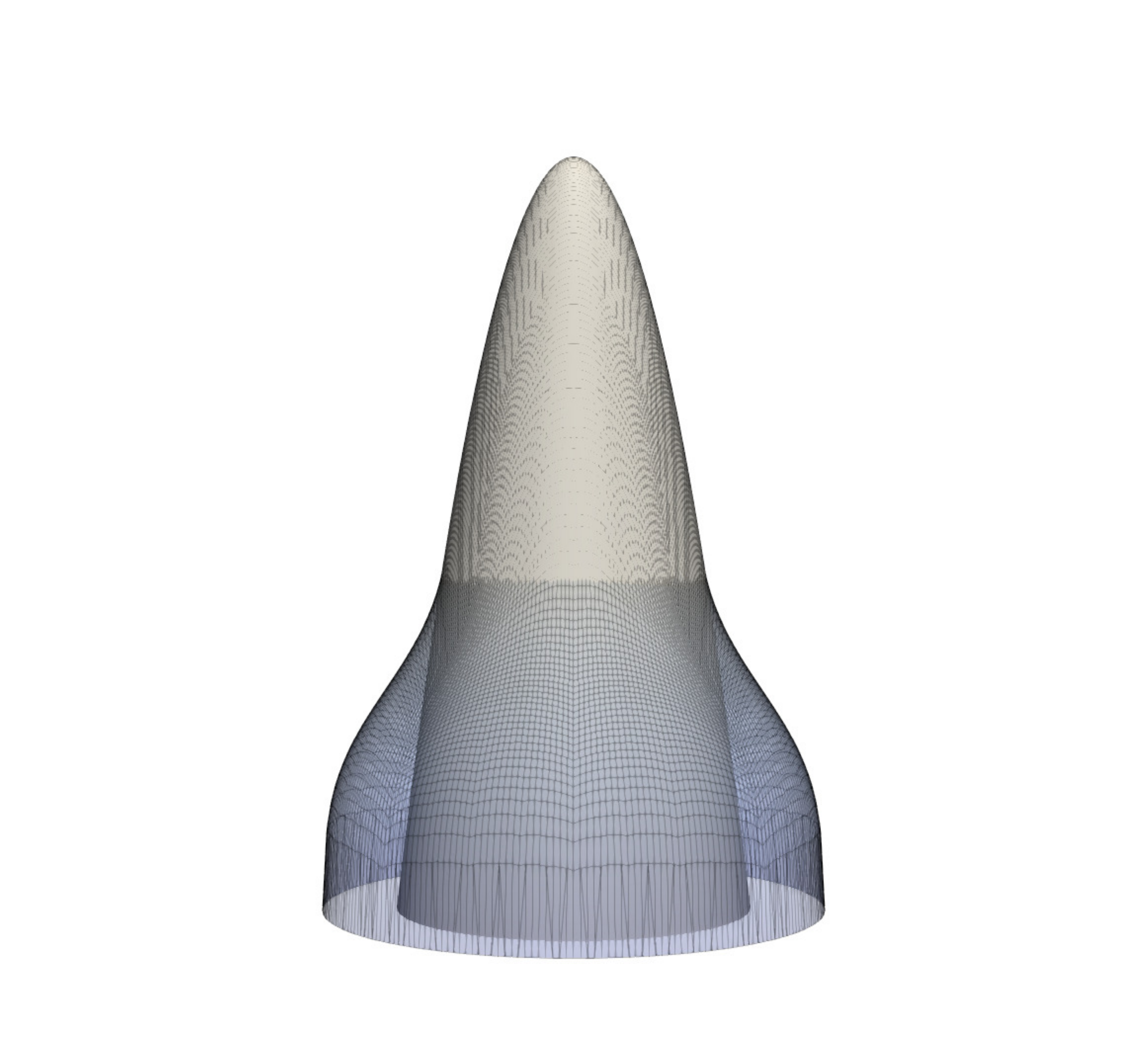} 
    & \!\!\!\!\!\!\!\!\!\!\!\!\!\includegraphics[scale = 0.1]{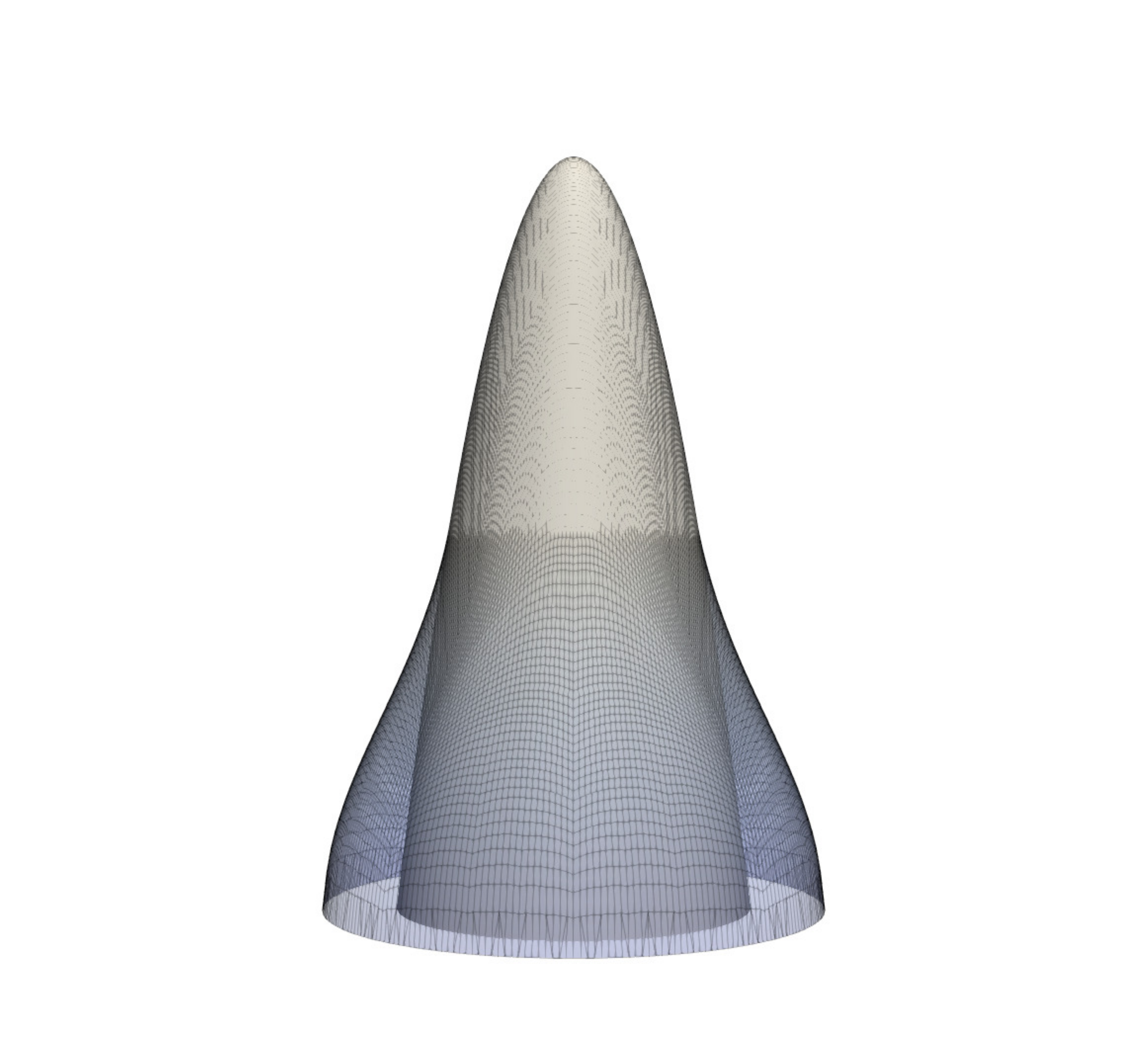}
     & \!\!\!\!\!\!\!\!\!\!\!\!\!\includegraphics[scale = 0.1]{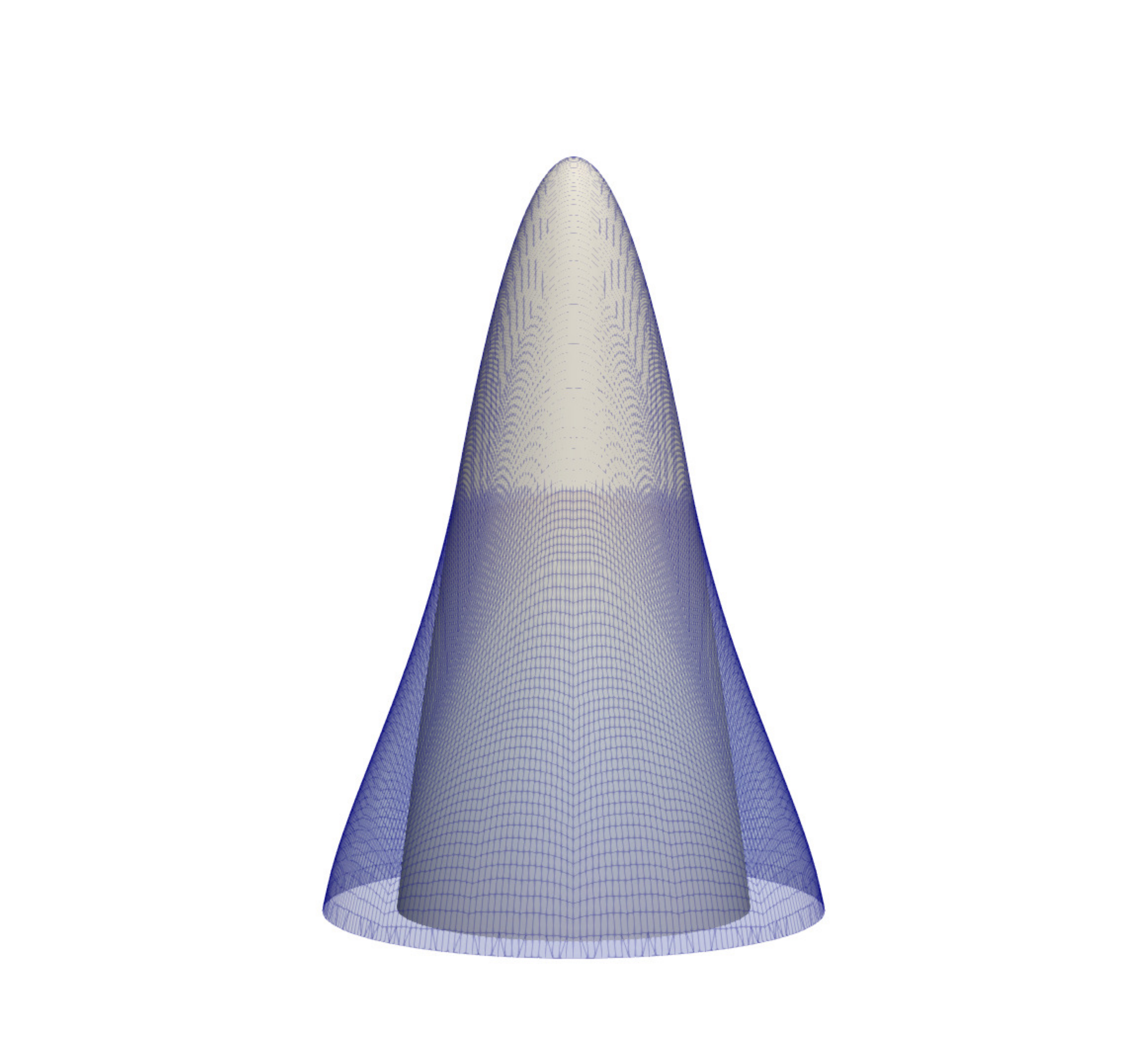} 
    \end{tabular}
  \end{center}
\caption{Case~\ref{case:frac}: Pure fractional diffusion case in the unit ball. Plot of $u_{h_6}$ for $s=0.3$ (left), $s=0.5$ (mid), $s=0.7$ (right). }
\end{figure}

\begin{figure}[hbt!]
\label{fig:2ddrift}
  \begin{center}
    \begin{tabular}{ccc}
      \!\!\!\!\!\!\! \!\!\!\!\!\!\! \!\!\!\!\!\!\!\includegraphics[scale = 0.2]{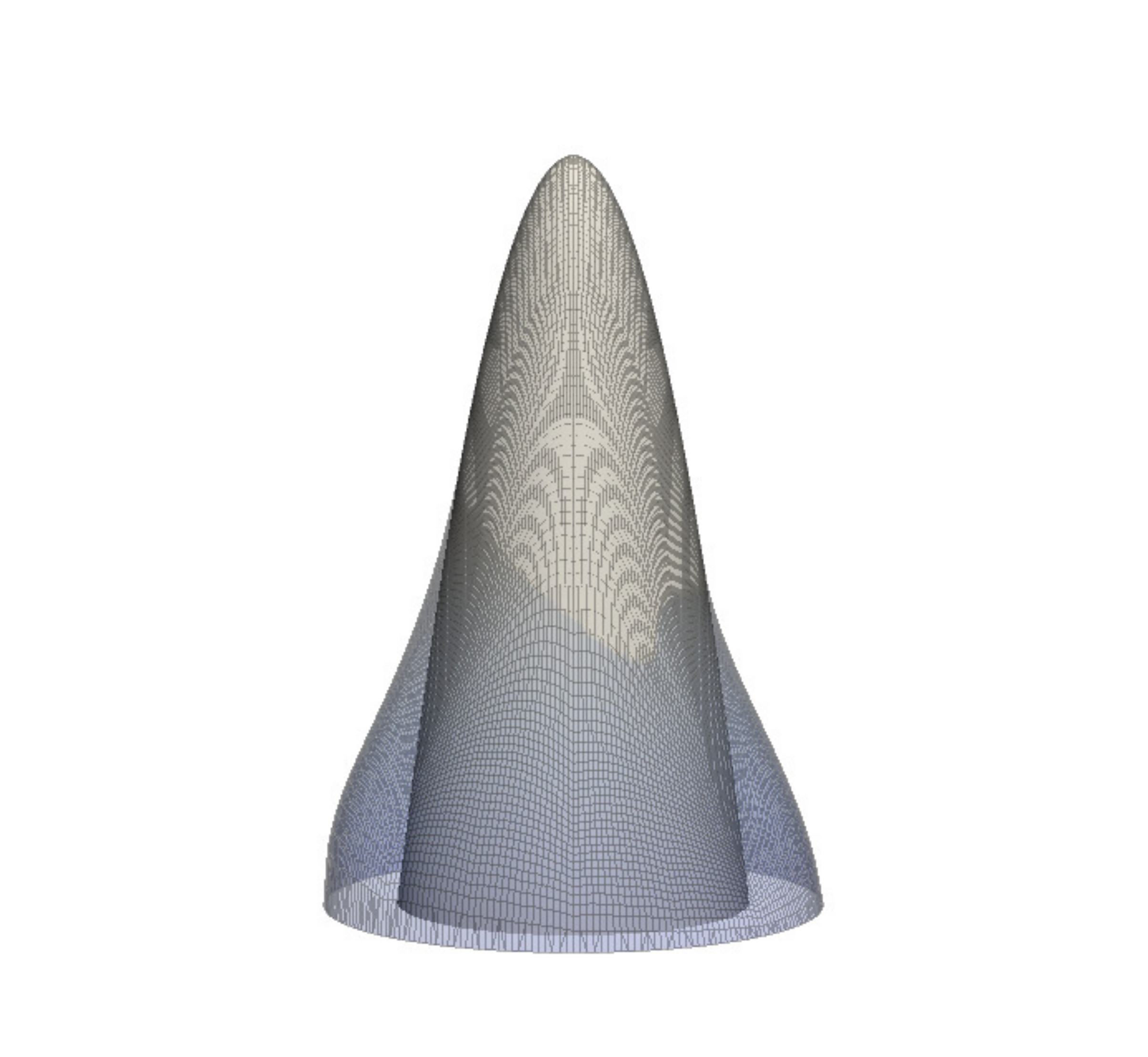} 
      & \!\!\!\!\!\!\!\!\!\!\!\!\!\!\!\!\!\! \!\!\!\!\!\!\!\includegraphics[scale = 0.2]{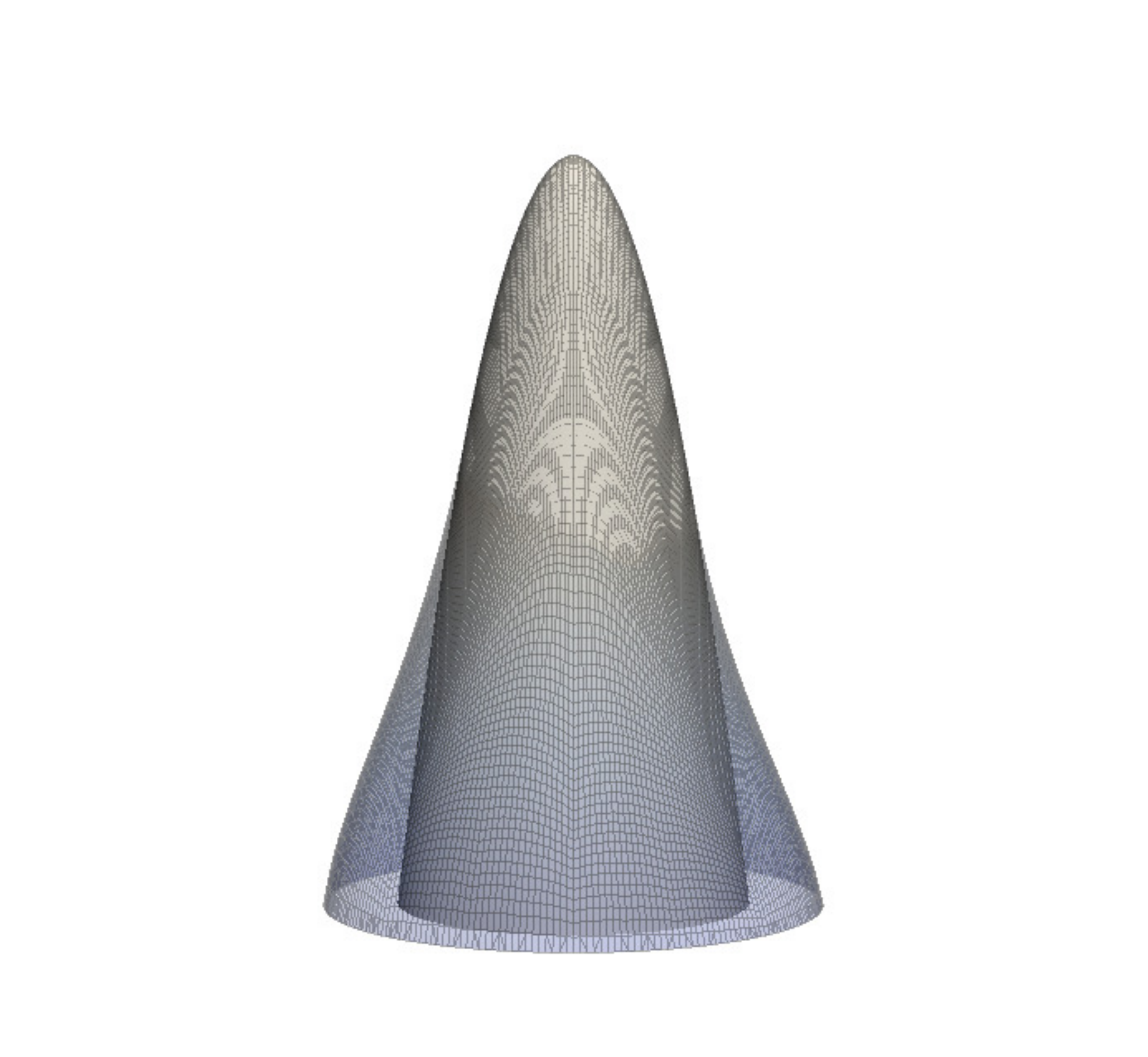} &
     \!\!\!\!\!\!\! \!\!\!\!\!\!\! \includegraphics[scale = 0.2]{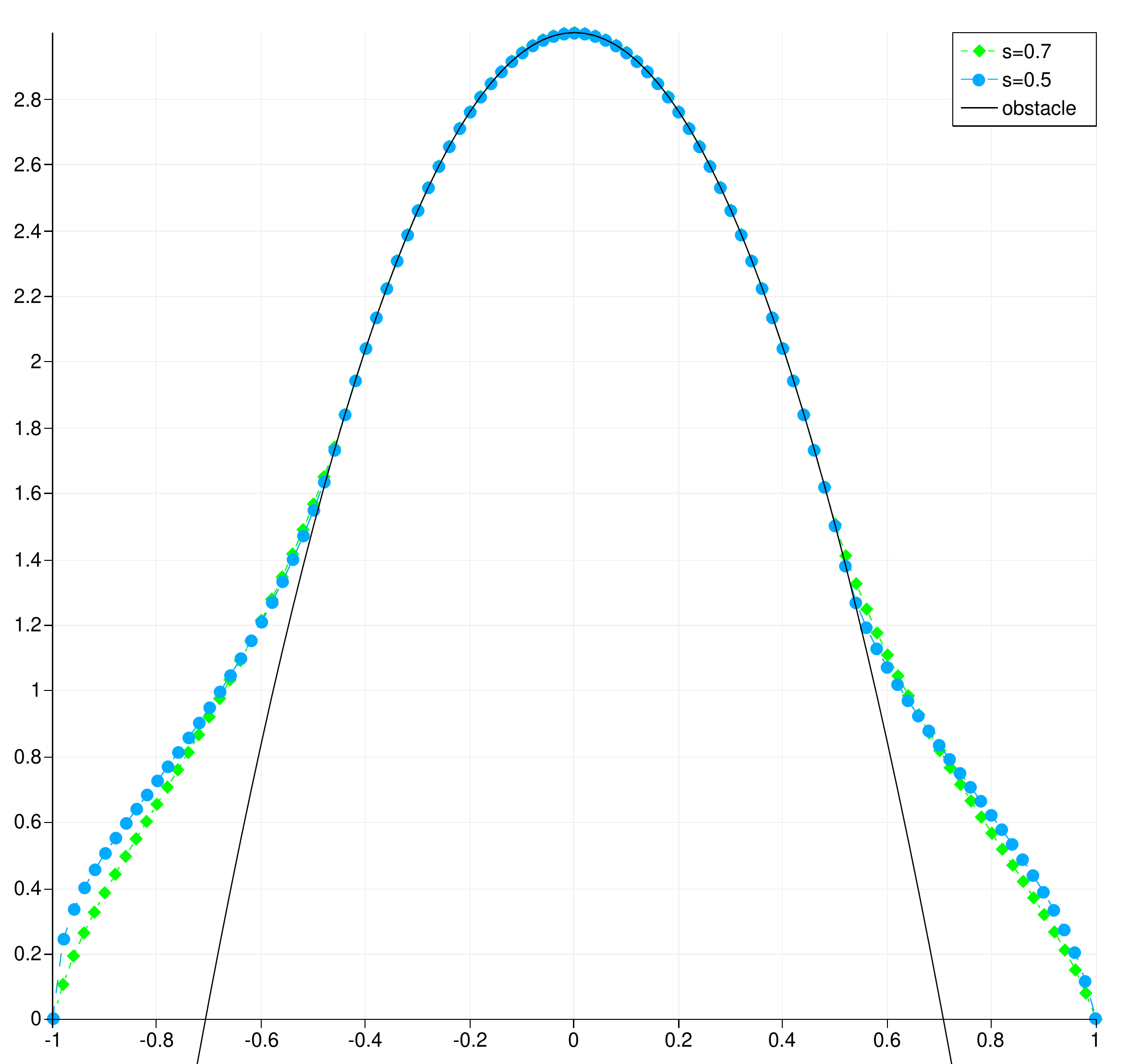}
        \end{tabular}
  \end{center}
\caption{Case~\ref{case:fracdrift}: Fractional diffusion case with drift. Plot of $u_{h_6}$ for $s=0.5$ (left). Plot of the solution for $s=0.7$ (mid). Cut along the $x$-axis (right). Note that the case $s=0.5$ is not included in the theory developed here. }
\end{figure}

\begin{figure}[hbt!]
\label{fig:2dintegrodifferential}
  \begin{center}
    \begin{tabular}{cccc}
       \includegraphics[scale = 0.15]{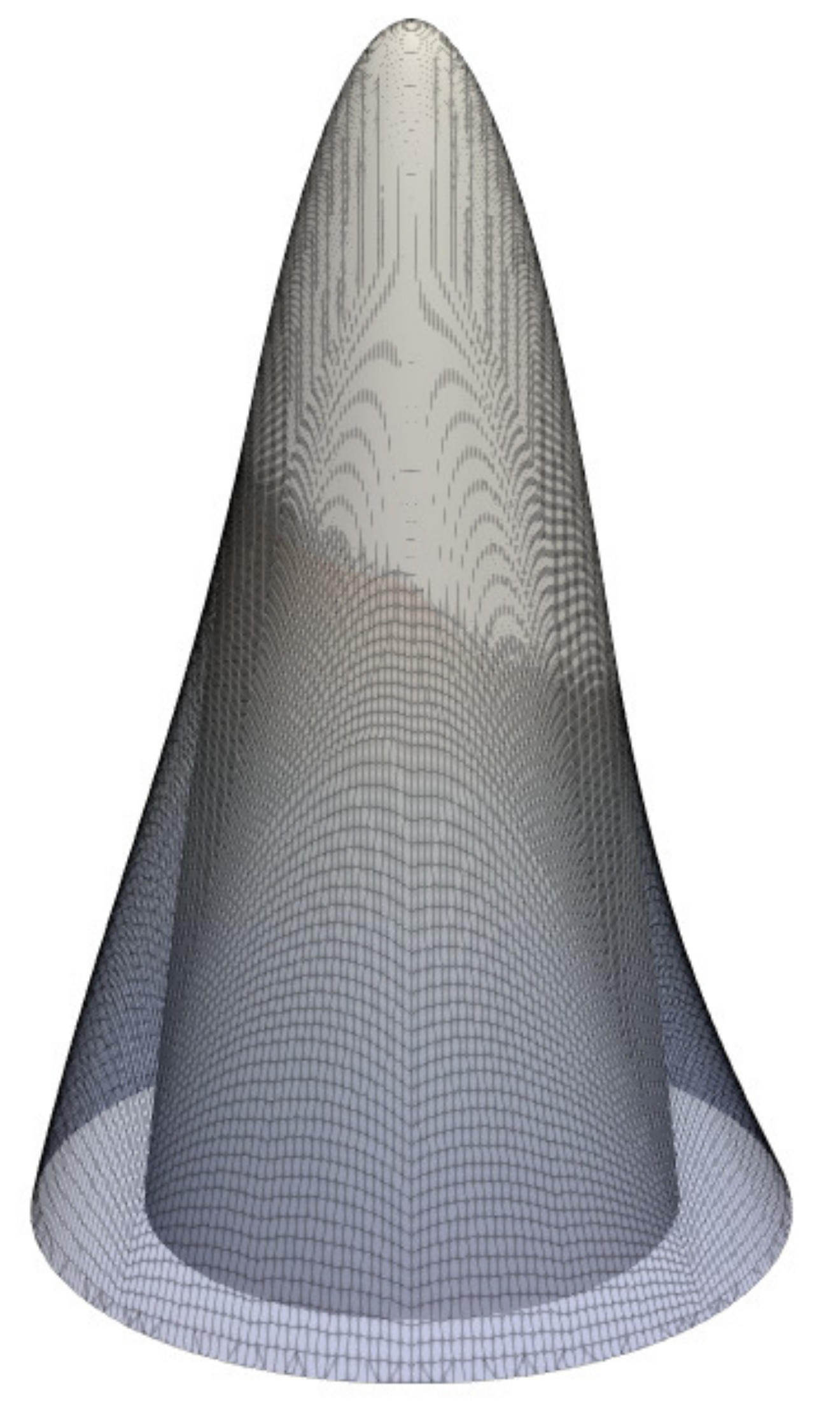} 
      &  \includegraphics[scale = 0.15]{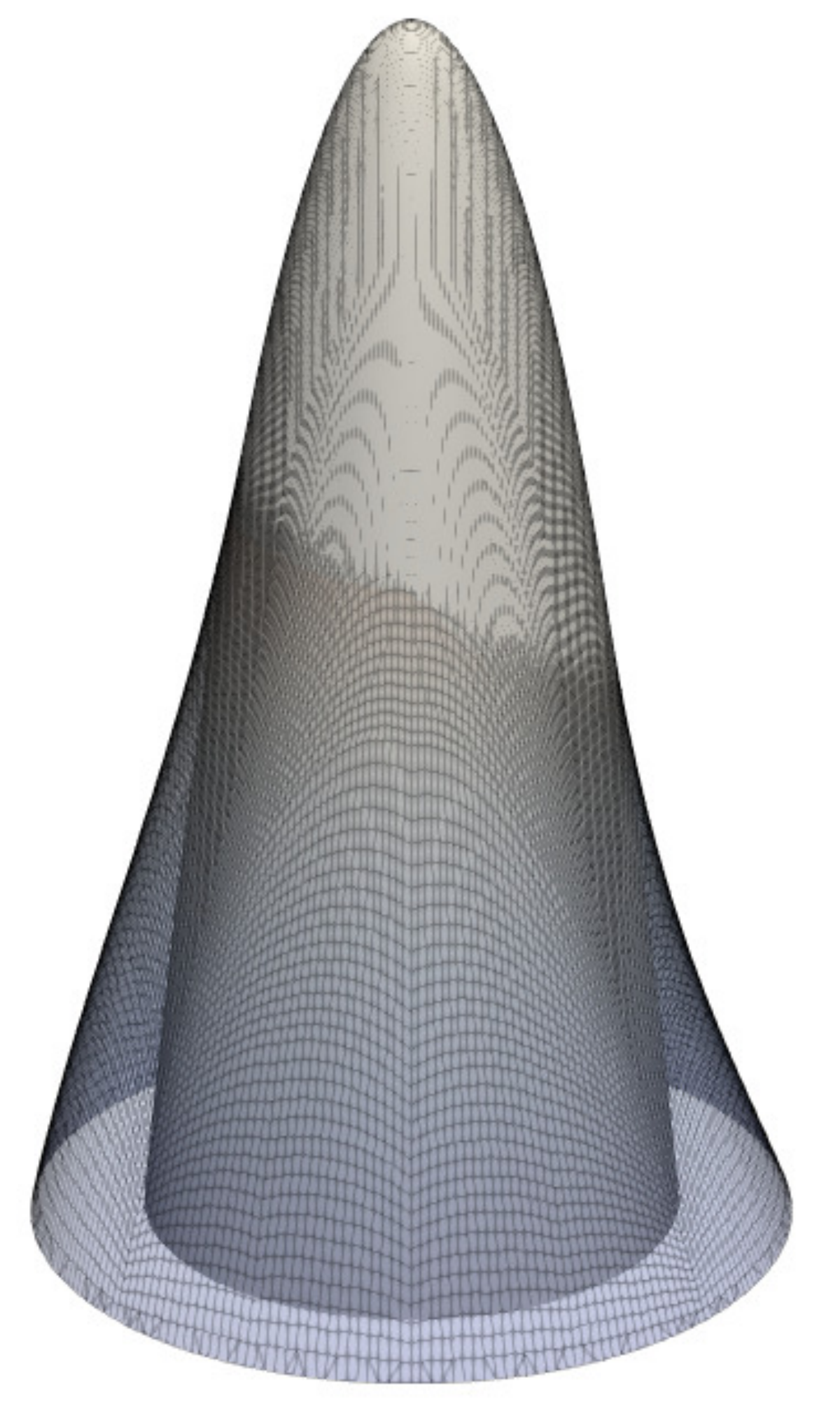} 
     & \includegraphics[scale = 0.15]{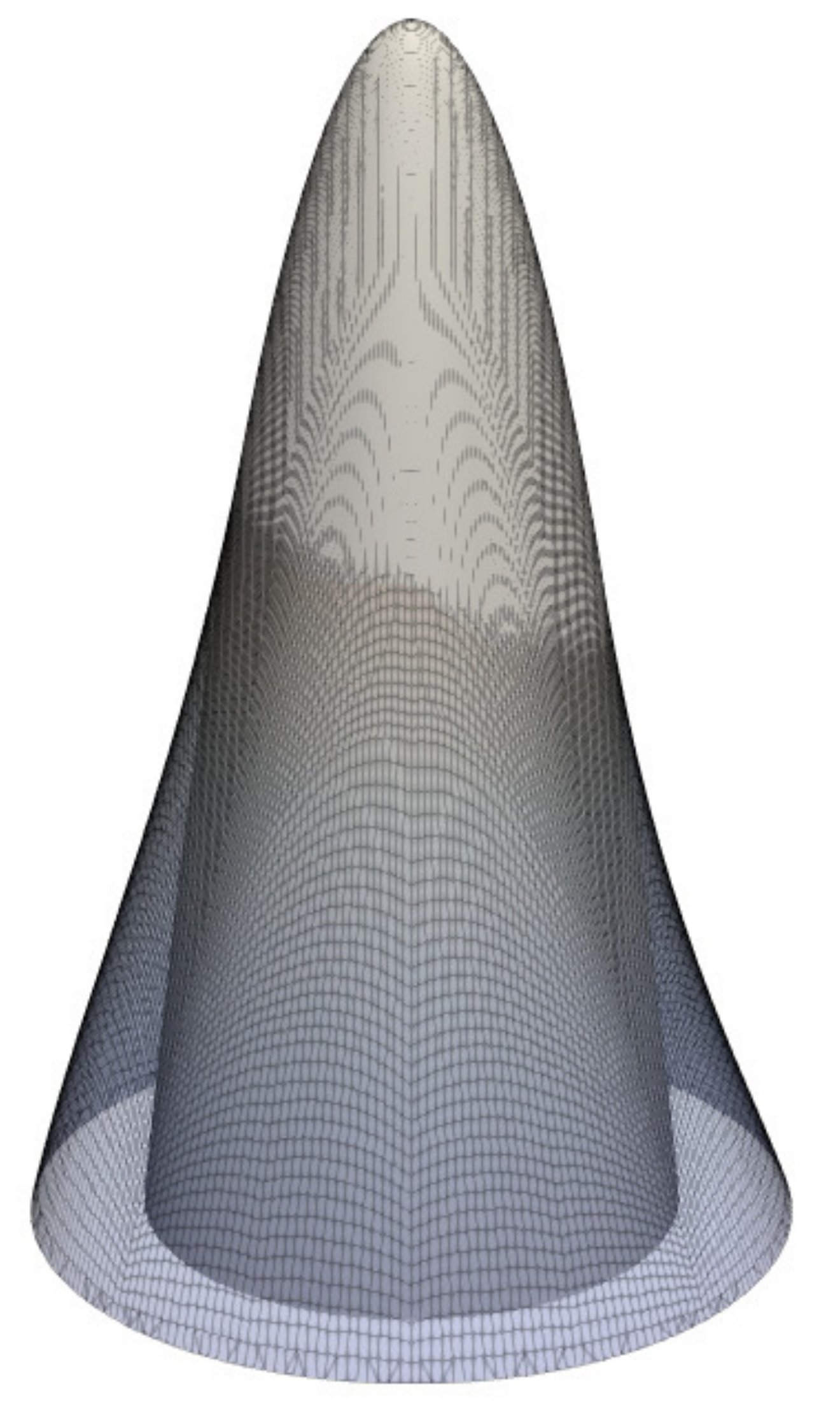} 
     & \includegraphics[scale = 0.2]{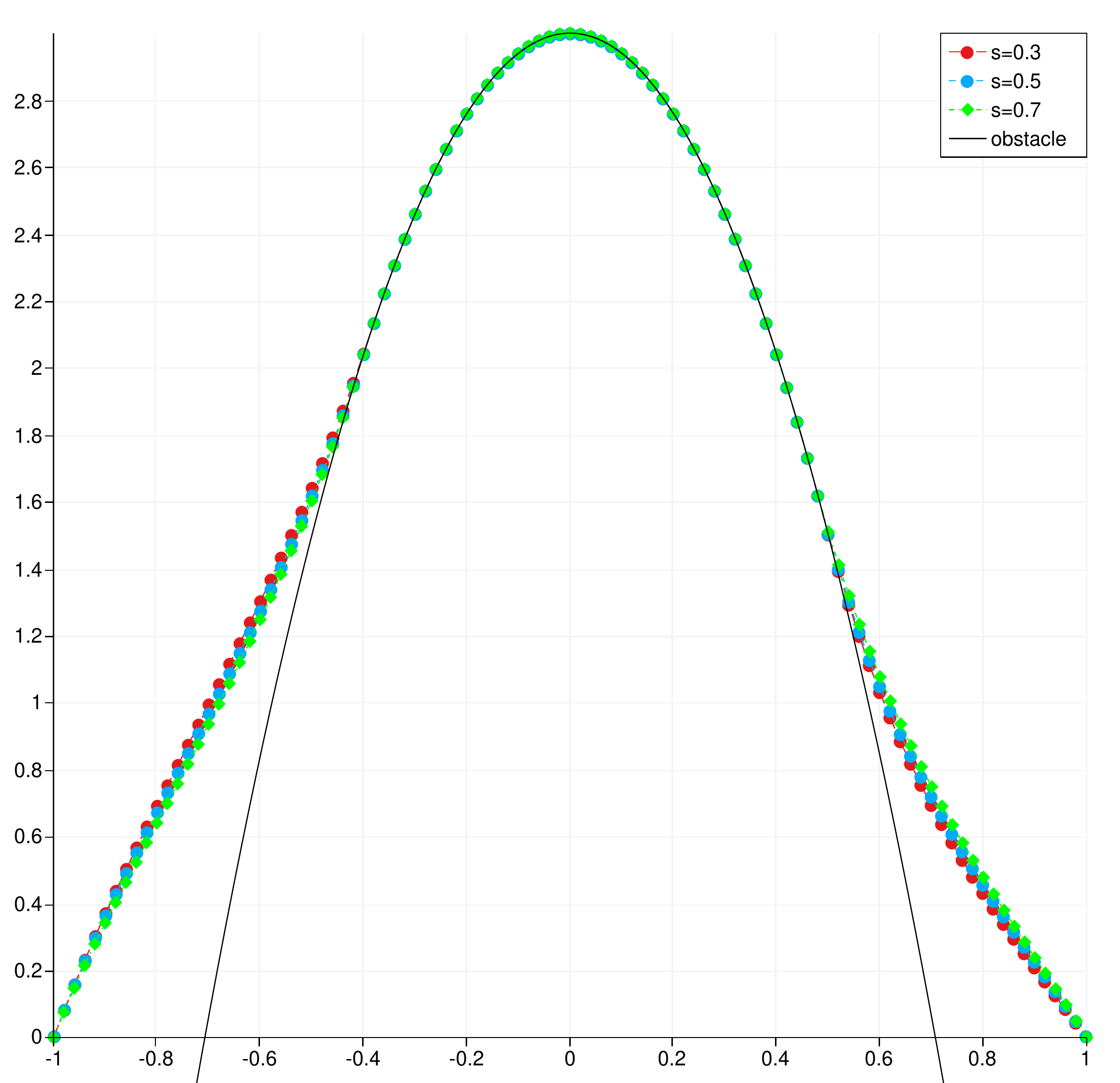}
    \end{tabular}
  \end{center}
\caption{Case~\ref{case:fracdrift}: Integro--differential case in the unit ball.  From left to right: plot of $u_{h_6}$ for $s=0.3$, $s=0.5$, $s=0.7$. Right: cut of along the $x$-axis.}
\end{figure}

The errors are computed using an overrefined solution $u_{\text{ref}}=u_{h_6}$ and we report the observed rate of convergence $\text{OROC} := \log(e_2/e_3)/\log(2)$ in Table~\ref{tab:2d-rate}. We note that the pure fractional diffusion case exhibits an observed the rate of convergence of $\mathcal O(h^{0.6})$, slightly better than predicted while for the other two cases matches the predictions of Theorem~\ref{thm:convergence_rate}.

\begin{table}[hbt!]
  \begin{center}
    \begin{tabular}{|c|c|c|c|}
     \hline
      & $s=0.3$ & $s=0.5$ & $s=0.7$ \\ \hline
      Case~\ref{case:frac} &  $0.57$ & $0.60$ & $0.67$ \\ \hline
      Case~\ref{case:fracdrift} & N/A & $0.59$ & $0.70$ \\ \hline
      Case~\ref{case:diff} & $1.00$ & $0.97$ & $0.89$ \\ \hline
    \end{tabular}
 \end{center}
  \caption{OROC for different cases and different values of the fractional power $s$.}
 \label{tab:2d-rate}
\end{table}

\subsubsection{L--shaped domain}

We now focus our attention to non--smooth domains and consider the standard L--shaped domain, \ie
$\Omega = (-\tfrac12,\tfrac12)^2\setminus(0,\tfrac12)^2$. We set $\chi(x,y) = 16^2x(x+\tfrac12)y(y-\tfrac12)$, and $f\equiv 1$. 
We consider the following two settings:
\begin{enumerate}[$\bullet$]
  \item Case~\ref{case:frac}, pure fractional diffusion in a non--smooth domain: Despite the fact that the theory developed in this work requires smooth domains, we provide numerical observations in Figure~\ref{fig:Lshape}.
  
  \item Case~\ref{case:diff}, integro--differential case: $A=0.3 \calI$, $c \equiv 0$, and $\bbeta = \mathbf 0$.
  The numerical results are gathered in Figure~\ref{fig:2dintegrodifferential-Lshaped}.
\end{enumerate}

The coarse subdivision of $\Omega$ consists of 12 squares each of diameter $\sqrt{2}/4$. Uniform refinements are performed to create a sequence of meshes $\mathcal T_{h_j}$, $j \geq 1$. 

\begin{figure}[hbt!]
\label{fig:Lshape}
  \begin{center}
    \begin{tabular}{cccc}
       \includegraphics[scale = 0.1]{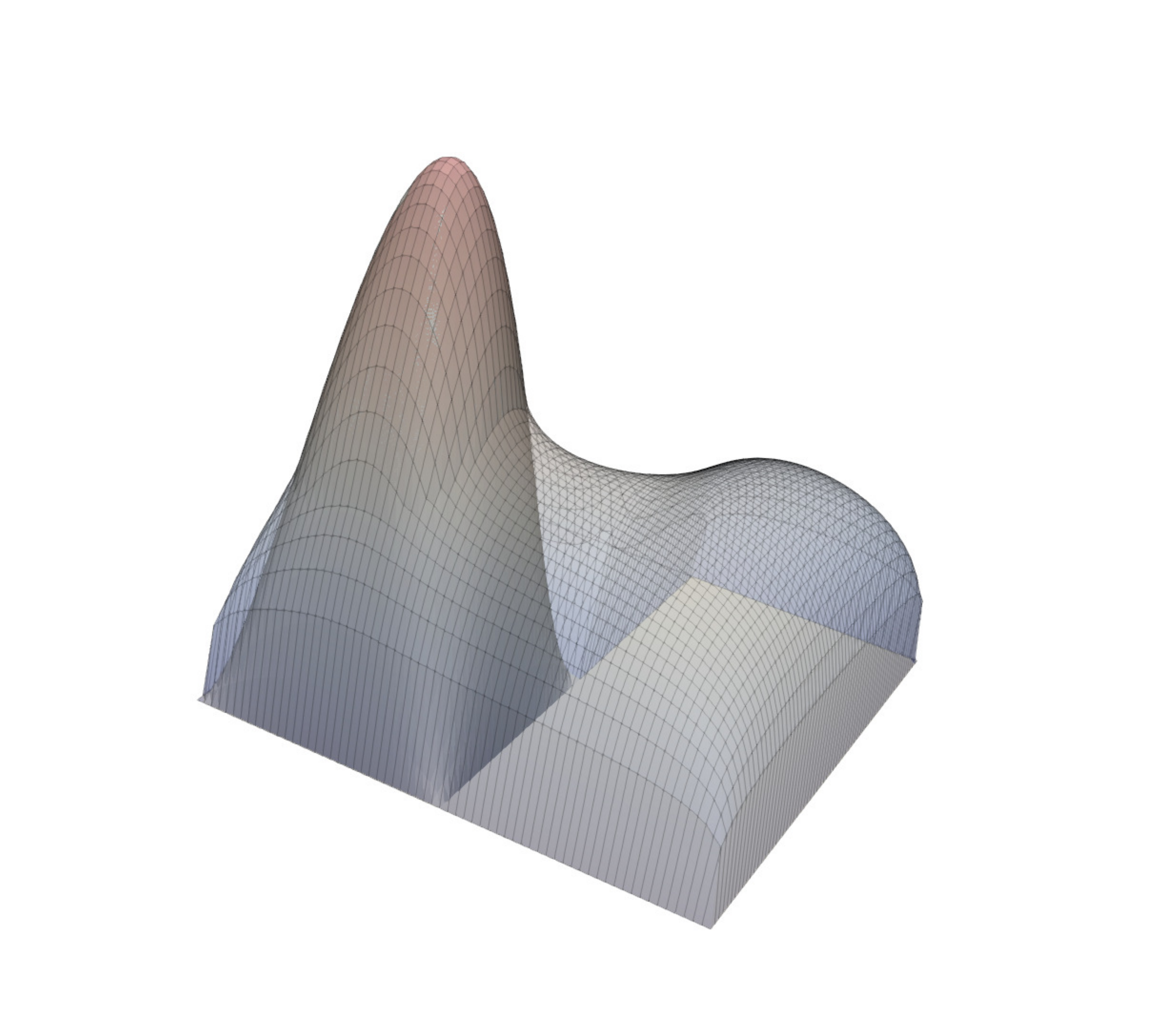} 
      &  \!\!\!\!\!\!\!\!\!\!\!\!\!\includegraphics[scale = 0.1]{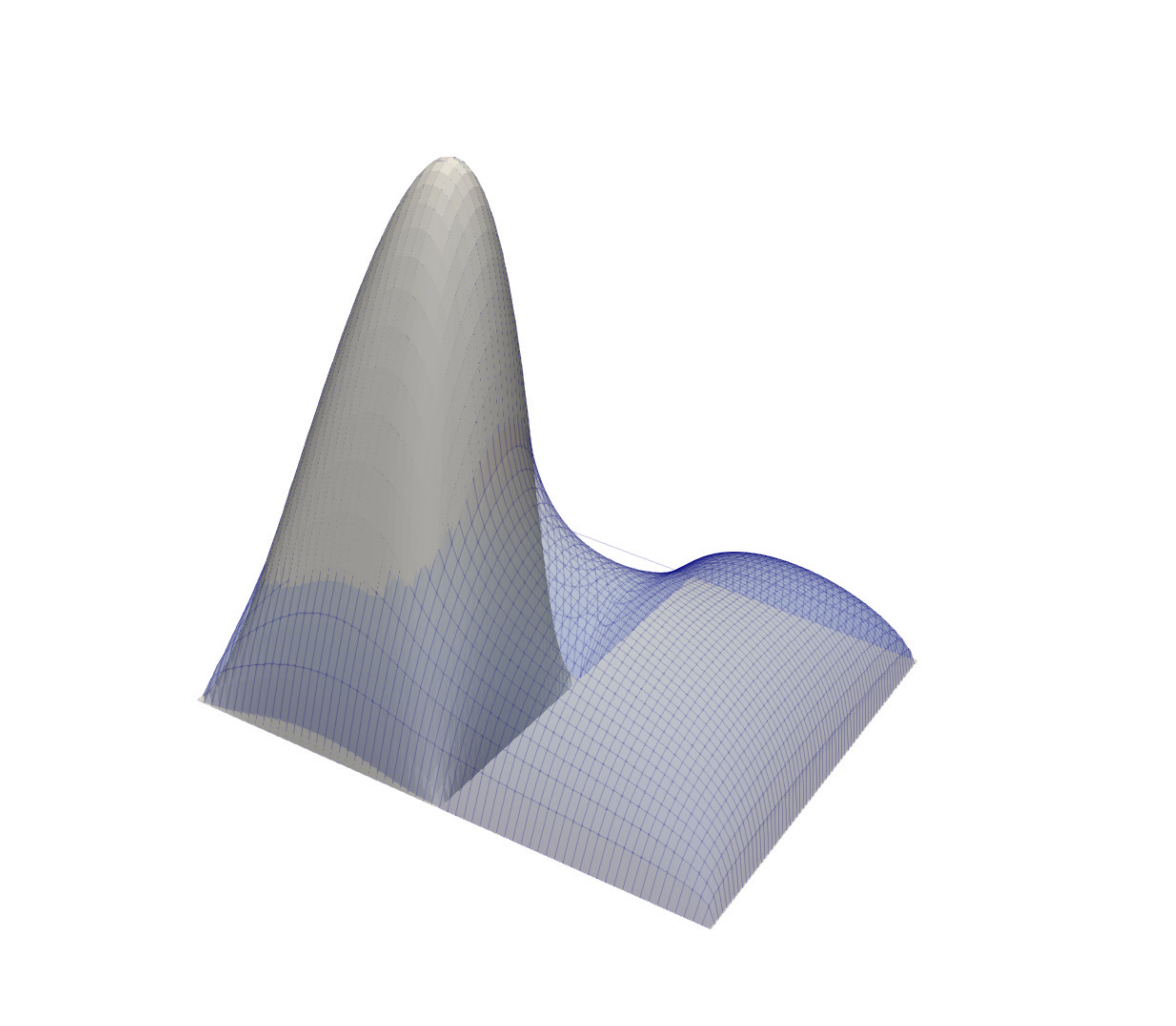} 
     & \!\!\!\!\!\!\!\!\!\!\!\!\!\includegraphics[scale = 0.1]{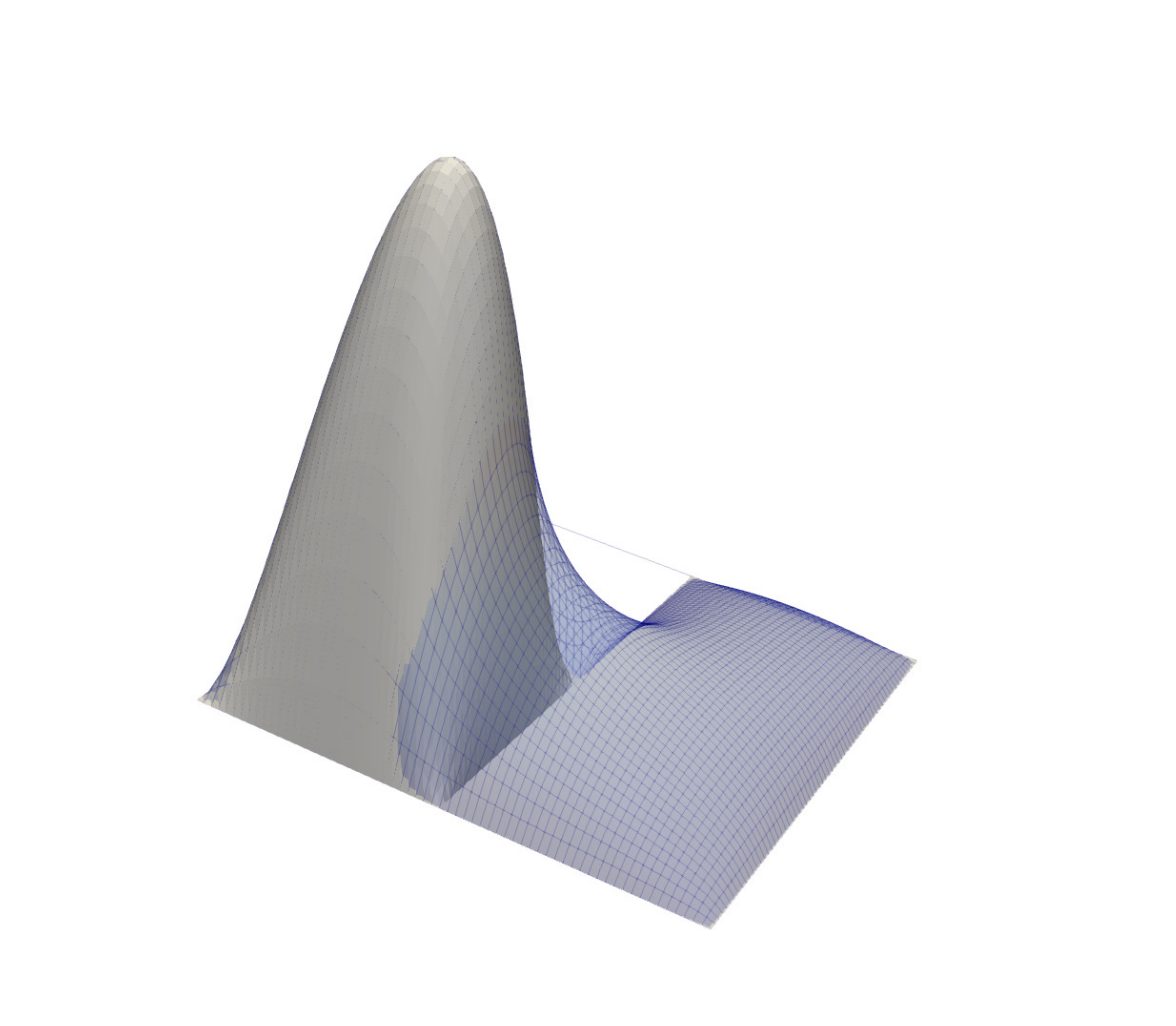} 
    \end{tabular}
  \end{center}
\caption{Solution for pure fractional diffusion case in a L--shaped domain. Plot of the solution for $s=0.3$ (left), $s=0.5$ (mid), $s=0.7$ (right). }
\end{figure}

\begin{figure}[hbt!]
\label{fig:2dintegrodifferential-Lshaped}
  \begin{center}
    \begin{tabular}{cccc}
       \includegraphics[scale = 0.15]{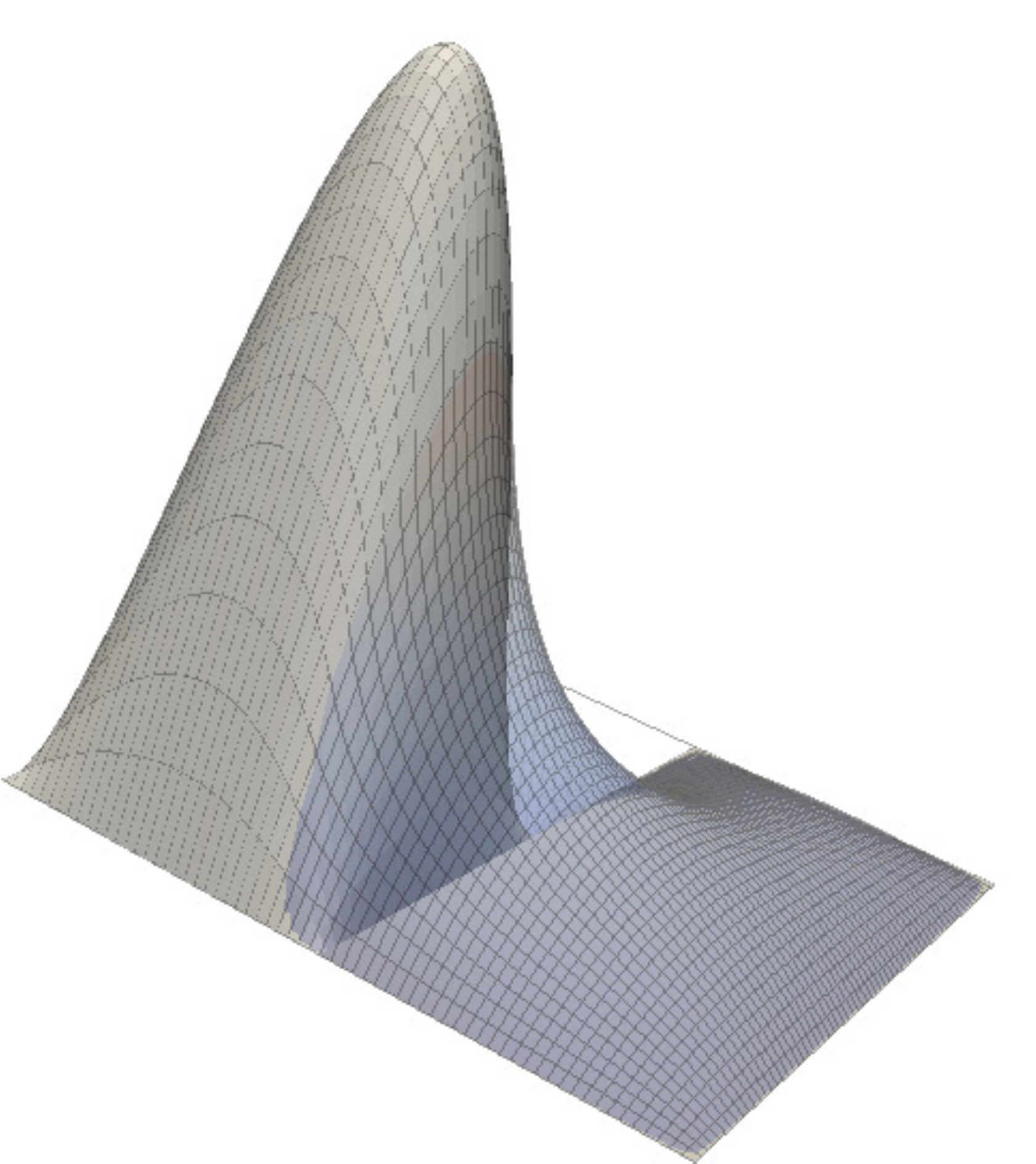} 
      & \!\!\!\!\!\!\!\!\! \includegraphics[scale = 0.15]{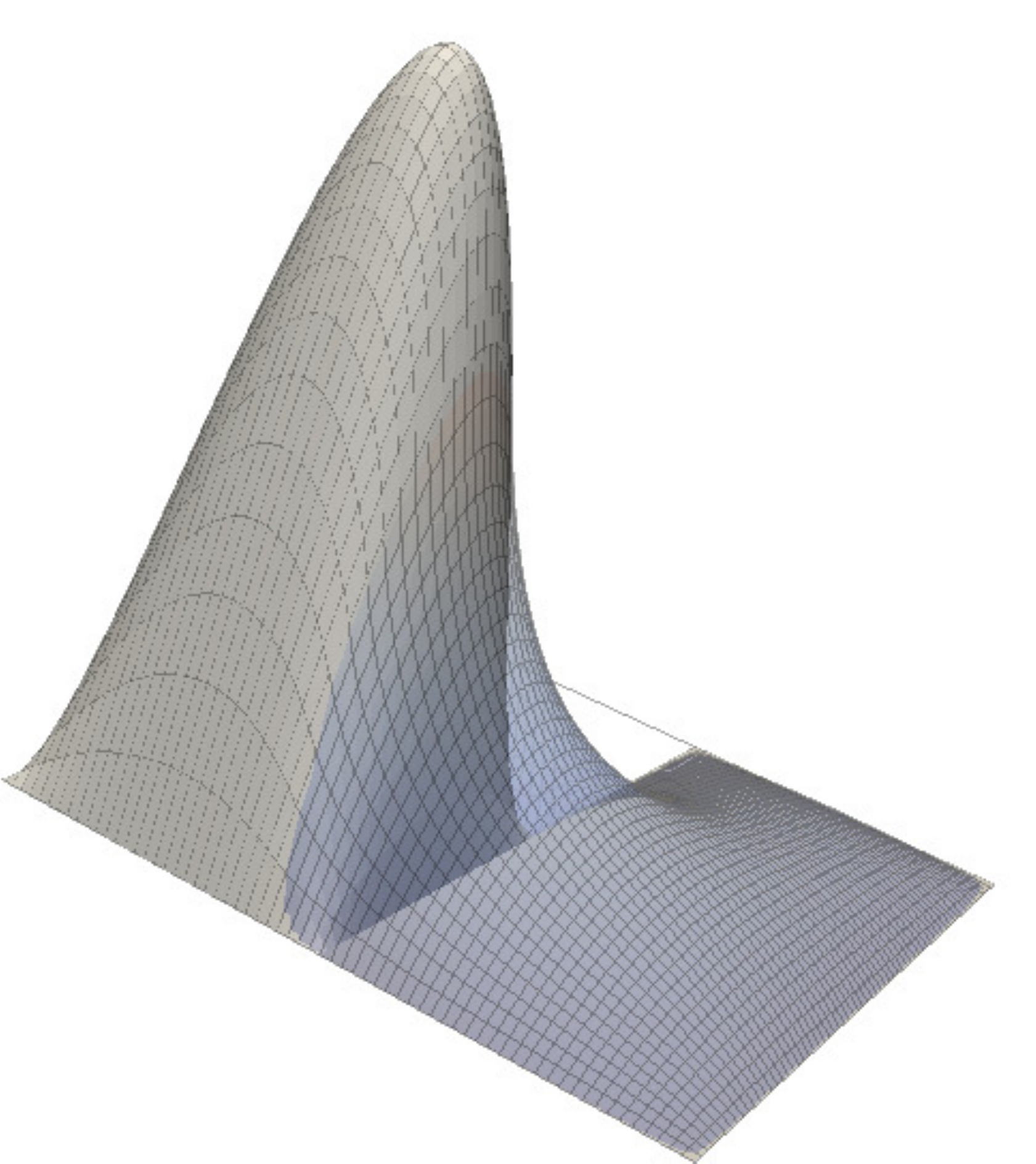} 
     & \!\!\!\!\!\!\!\! \includegraphics[scale = 0.15]{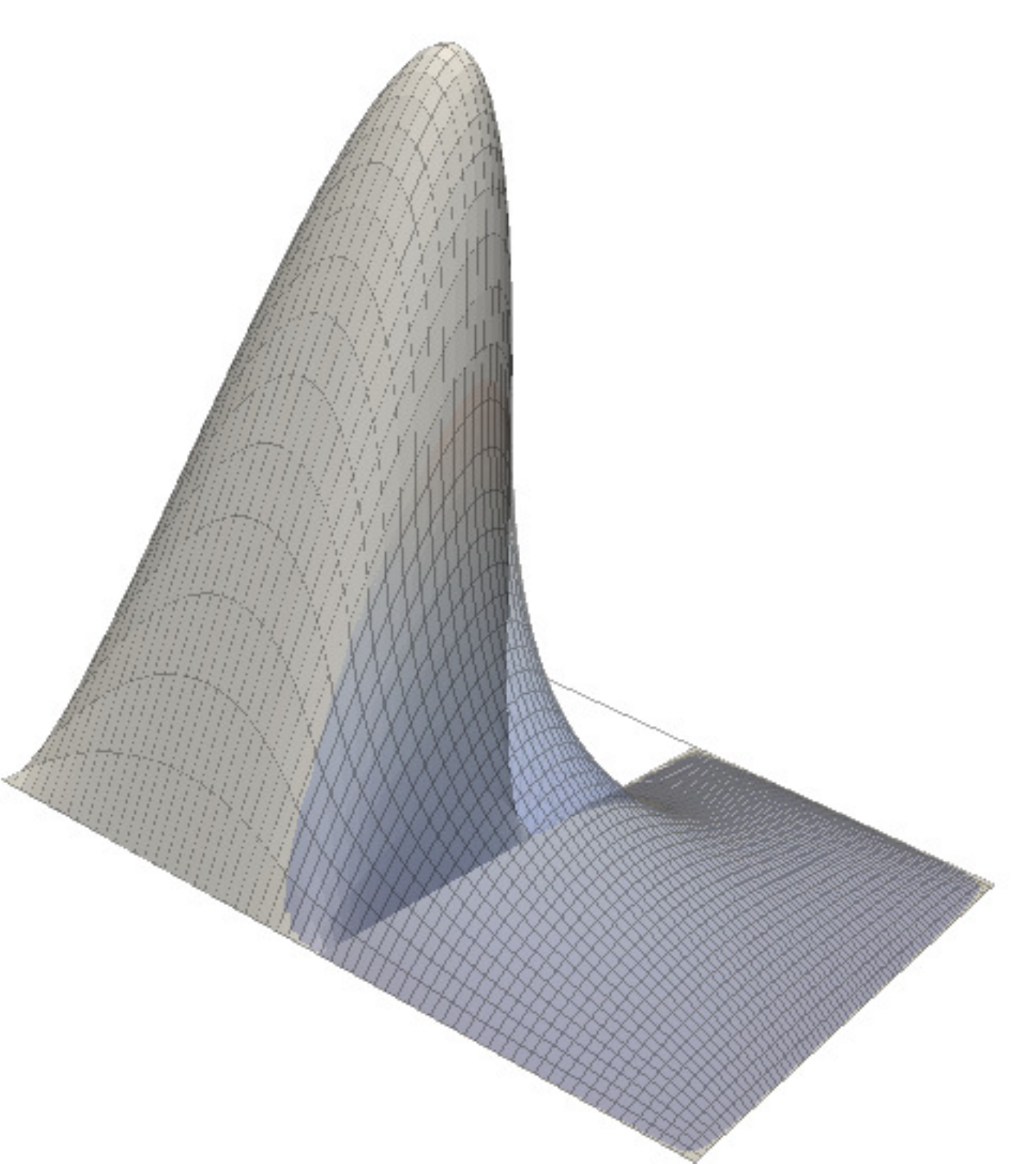} 
     &  \!\!\!\!\!\! \includegraphics[scale = 0.16]{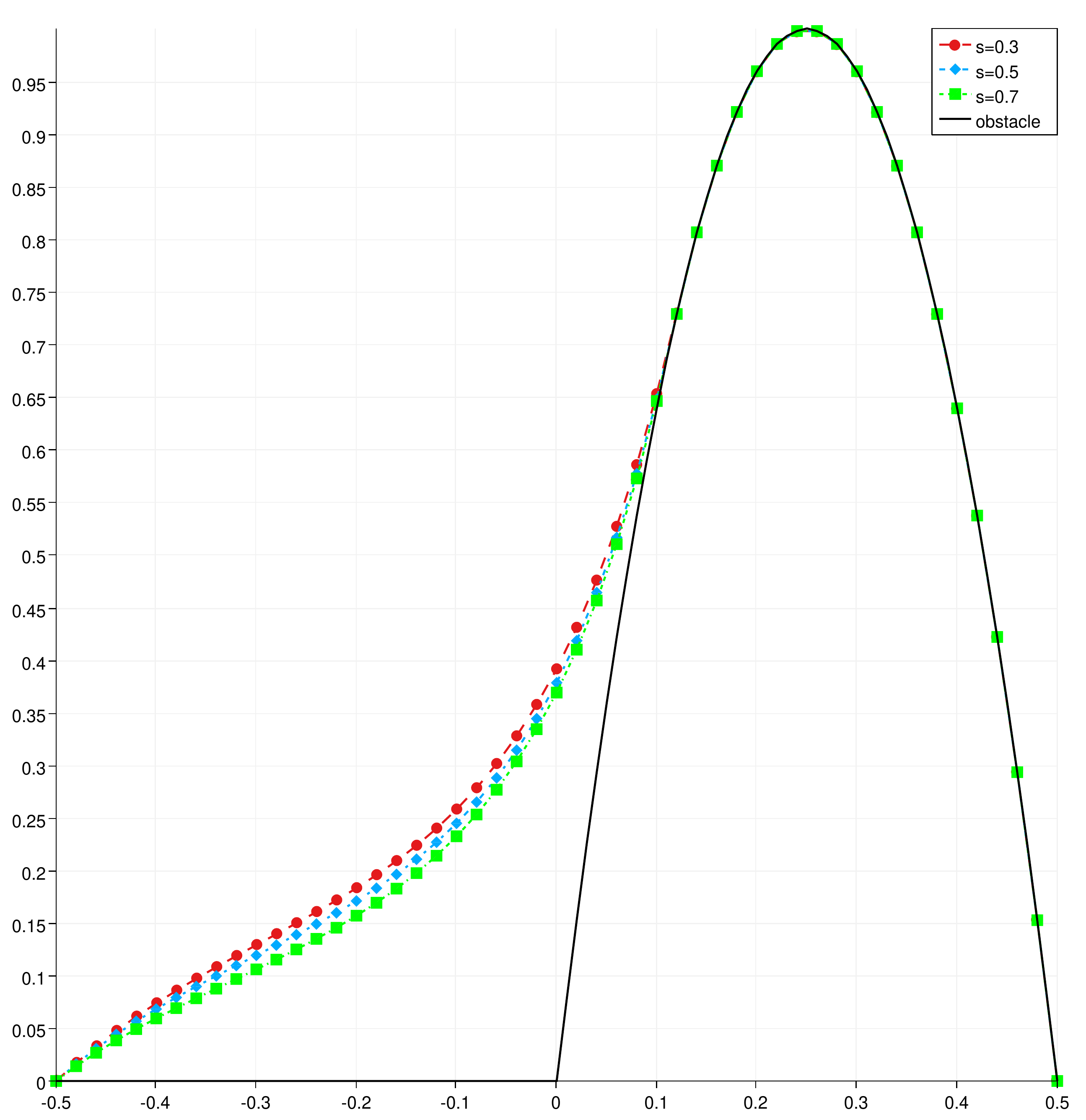}
    \end{tabular}
  \end{center}
\caption{Solution for the integro--differential case in a L--shaped domain. From left to right: plot of the solution for $s=0.3$, $s=0.5$, $s=0.7$. Right: section of the solution along $x=-0.25$.}
\end{figure}

The errors are computed using as reference solution $u_{\text{ref}}=u_{h_4}$.
We report the observed rate of convergence $\text{OROC} := \log(e_1/e_2)/\log(2)$ in Table~\ref{tab:2d-rate_2}. 
In all cases, the observed rate of convergence is better than the prediction given by Theorem~\ref{thm:convergence_rate}. We suppose that this is due to the use of a finer approximate solution to estimate the error.

\begin{table}[hbt!]
  \begin{center}
    \begin{tabular}{|c|c|c|c|}
     \hline
      & $s=0.3$ & $s=0.5$ & $s=0.7$ \\ \hline
      Case A & $0.66$ & $1.09$ & $1.29$ \\ \hline
      Case B & $1.00$ & $1.01$  & $1.02$ \\ \hline
    \end{tabular}
 \end{center}
  \caption{OROC for different cases and different values of the fractional power $s$.}
 \label{tab:2d-rate_2}
\end{table}

\bibliographystyle{plain}

\end{document}